\documentclass[a4paper, 11pt, twoside, notitlepage]{amsart}

\usepackage{amsmath,amscd}
\usepackage{amssymb}
\usepackage{amsthm}
\usepackage{comment}
\usepackage{graphicx, xcolor}

\usepackage{mathrsfs}
\usepackage[colorlinks=true, linkcolor=blue]{hyperref}
\usepackage[normalem]{ulem}

\newcommand\redsout{\bgroup\markoverwith{\textcolor{red}{\rule[0.5ex]{2pt}{0.4pt}}}\ULon}

\usepackage{bm}
\usepackage{bbm}
\usepackage{url}

\usepackage[utf8]{inputenc}
\usepackage{mathtools,amssymb}
\usepackage{esint}
\usepackage{tikz}
\usetikzlibrary{arrows.meta,calc}
\usepackage{dsfont}
\usepackage{relsize}
\usepackage{url}
\usepackage{xcolor}
\usepackage{graphicx}
\usepackage{mathrsfs}
\usepackage[shortlabels]{enumitem}
\usepackage{lineno}
\usepackage{amsmath}
\usepackage{enumitem}
\usepackage{amsthm} 
\usepackage{verbatim}
\usepackage{dsfont}
\numberwithin{equation}{section}

\allowdisplaybreaks
\newcommand{\para}[1]{\vspace{3mm} \noindent\textbf{#1.}} 

\mathtoolsset{showonlyrefs}

\graphicspath{{images/}}

\newtheorem{theorem}{Theorem}[section]
\newtheorem{lemma}[theorem]{Lemma}
\newtheorem*{lemma*}{Lemma}
\newtheorem{definition}[theorem]{Definition}
\newtheorem{corollary}[theorem]{Corollary}

\newtheorem{proposition}[theorem]{Proposition}

\newtheorem{remark}[theorem]{Remark}

\title[Single-measurement obstacle uniqueness]{A Schiffer-type inverse obstacle problem for the planar Schr\"odinger equation with a single measurement}

\author[T. Liimatainen]{Tony Liimatainen}
\address{Department of Mathematics and Statistics, University of Helsinki, Helsinki, Finland}
\curraddr{}
\email{tony.t.liimatainen@jyu.fi}

\author[Y.-H. Lin]{Yi-Hsuan Lin}
\address{Department of Applied Mathematics, National Yang Ming Chiao Tung University, Hsinchu, Taiwan \& Fakult\"at f\"ur Mathematik, University of Duisburg-Essen, Essen, Germany}
\curraddr{}
\email{yihsuanlin3@gmail.com}

\author[P. Ola]{Petri Ola}
\address{Department of Mathematics and Statistics, University of Helsinki, Helsinki, Finland}
\curraddr{}
\email{petri.ola@helsinki.fi}


\newcommand{\R}{{\mathbb R}}

\newcommand {\p} {\partial}

\begin{document}

\begin{abstract}
	We prove a single-measurement uniqueness result for interior sound-soft obstacles in the planar Schr\"odinger equation with an unknown surrounding potential. For any two configurations with different obstacles, we construct a boundary datum for which the corresponding Neumann measurements differ. The full Dirichlet-to-Neumann map uniquely determines both the obstacle and the potential. The obstacles may have several connected components, and no convexity or analyticity assumptions are imposed. Neither potential is required to be known near the measurement boundary.
\end{abstract}

	\maketitle

	\tableofcontents

	\section*{Disclaimer}
	This is a preliminary version, released rapidly because of the urgent pace of AI-driven developments in the field. The main results, arguments, and statements of this paper were developed without the use of AI. AI tools were used only for auxiliary mathematical discussions and for assistance with parts of the writing. The authors take full responsibility for the final mathematical statements and proofs. The authors will improve the presentation in an updated version following this initial release.

	\section{Introduction}\label{sec:introduction}
	
	In this paper, we investigate the simultaneous recovery of an interior sound-soft obstacle and its surrounding potential for the Schr\"odinger equation in a bounded planar domain. We prove that the Dirichlet-to-Neumann (DN) map on the outer boundary determines both unknowns when the potentials are real-valued and belong to $C^{1,\alpha}$, with $0<\alpha<1$. For each fixed pair of configurations with different obstacles, we construct one boundary datum for which the Neumann measurements differ. Neither potential is assumed to be known near the measurement boundary.
	
	Let $\Omega\subset\mathbb R^2$ be a bounded domain with $C^\infty$ boundary, and let $D\Subset\Omega$ be an obstacle with smooth boundary. The precise class of obstacles is given below. Throughout the paper, the outer boundary means the whole of $\partial\Omega$, including all its connected components. For a real-valued potential $q\in L^\infty(\Omega\setminus\overline D)$, consider
	\begin{equation}\label{eq:main}
		\begin{cases}
			(-\Delta+q)u=0&\text{in }\Omega\setminus\overline D,\\
			u=0&\text{on }\partial D,\\
			u=f&\text{on }\partial\Omega.
		\end{cases}
	\end{equation}
	We assume that
	\begin{equation}\label{eq:eigenvalue_condition}
		\text{$0$ is not a Dirichlet eigenvalue of $-\Delta+q$ in $\Omega\setminus\overline D$.}
	\end{equation}
	By the Fredholm alternative, every $f\in H^{1/2}(\partial\Omega)$ gives a unique solution $u_f\in H^1(\Omega\setminus\overline D)$ of \eqref{eq:main}. The associated DN map is
	\begin{equation}\label{eq:DN_map}
		\Lambda_{D,q}:H^{1/2}(\partial\Omega)\to H^{-1/2}(\partial\Omega),\quad f\mapsto\partial_\nu u_f\big|_{\partial\Omega},
	\end{equation}
	where $\nu$ is the exterior unit normal to $\Omega$ and the normal derivative is understood in the weak sense. The inverse problem is the following:
	\begin{enumerate}[\textbf{(IP)}]
		\item\label{Q:IP} Does the DN map $\Lambda_{D,q}$ determine the obstacle $D$ and the surrounding potential $q$ simultaneously?
	\end{enumerate}
	
	\begin{definition}\label{def:admissible_obstacle}
		Let $\Omega\subset\mathbb R^2$ be a bounded domain with $C^\infty$ boundary. We say that $D\Subset\Omega$ is an \emph{admissible obstacle} if
		\begin{enumerate}[\rm(i)]
			\item\label{item 1 def adm} $D$ is a bounded open set,
			\item\label{item 2 def adm} $\partial D$ is of class $C^\infty$ and has finitely many connected components,
			\item\label{item 3 def adm} $\Omega\setminus\overline D$ is connected.
		\end{enumerate}
	\end{definition}
	
	We assume that $\Omega\setminus\overline D$ is connected, but $D$ itself need not be connected. We write $D=\bigcup_{k=1}^N H_k$, where $H_1,\ldots,H_N$ are its pairwise disjoint connected components. Neither their number nor their shapes are prescribed.
	
	\para{Earlier works} A classical question in inverse scattering is whether a sound-soft obstacle can be determined from the far-field pattern produced by one incident plane wave at a fixed wave number. This is commonly called Schiffer's problem; see \cite{colton2018looking,blasten2021scattering}. Uniqueness from all incident directions at a fixed wave number is recorded in \cite{lax1967scattering}. For smooth sound-soft obstacles, uniqueness from one incident plane wave was proved under a size restriction relative to the wave number in \cite{colton1983uniqueness}, and this restriction was relaxed in \cite{gintides2005local}. Local uniqueness near an arbitrary smooth sound-soft obstacle, at any given wave number and incident direction, was established in \cite{stefanov2004local}. For sound-soft polyhedral scatterers, uniqueness from one incident plane wave was proved in \cite{alessandrini2005determining}. For each fixed pair of distinct sound-soft obstacles, \cite{liu2024solving} proves that their far-field patterns differ almost surely when the wave number or incident direction is chosen at random. Geometric uniqueness and approximate uniqueness results for source and medium supports were obtained under size and curvature assumptions in \cite{blasten2021scattering}.
	
	For a known planar domain, uniqueness of the potential from full Cauchy data for the Schr\"odinger equation was established in \cite{bukhgeim2008recovering}. For real-valued $C^{1,\alpha}$ potentials on a compact Riemann surface with boundary, uniqueness from Cauchy data on an arbitrary nonempty open subset of the boundary was proved in \cite{guillarmou2011calderon}.
	
	Obstacle recovery in a known inhomogeneous background was studied in \cite{kirsch1998recovering,nachman2007imaging}. For a known planar inhomogeneous and anisotropic medium, \cite{rondi2003uniqueness} determines multiple sound-soft defects at low frequency from at most two suitably chosen boundary measurements. For known elliptic operators with coefficients analytic near the obstacle, \cite{honda2013analytic} reconstructs a sound-soft obstacle with a Lipschitz, nowhere real analytic boundary from one measurement.
	
	When the surrounding medium is also unknown, a two-dimensional simultaneous uniqueness result was proved in \cite{hahner1998uniqueness} for a simply connected sound-soft obstacle, using Cauchy data on a surrounding circle for all incident waves and an interval of wave numbers. For one fixed incident direction, \cite{liu2017recovery} determines a sound-soft or sound-hard obstacle from far-field patterns over a frequency interval and also recovers the surrounding medium under additional admissibility assumptions. At one wave number, \cite{liu2008recovery} proves simultaneous uniqueness in three dimensions from far-field patterns for all incident and observation directions, for an admissible class of partially buried obstacles with planar contacts with the unknown inhomogeneity. These geometric assumptions differ from those of a smooth obstacle contained in the region where the potential is unknown.
	
	Simultaneous recovery from local boundary data has also been proved under structural assumptions on the coefficients. In \cite{yang2021uniqueness}, the leading and zeroth-order coefficients of a reduced wave equation and the embedded obstacles are recovered in dimensions $n\ge2$ under piecewise constant coefficient assumptions. For the Schr\"odinger equation in three dimensions, \cite{wu2025boundary} determines all derivatives of a smooth potential on the measured part of the outer boundary in the presence of an unknown obstacle. An additional analyticity assumption on the potential gives global recovery of the potential, the obstacle, and its boundary condition.
	
	For anisotropic fractional Schr\"odinger equations, simultaneous recovery of an obstacle and a potential was proved in \cite{cao2019simultaneously}. Any nonzero exterior Dirichlet value, together with its exterior Neumann measurement, determines the sound-soft obstacle, while recovery of the potential uses the exterior DN map. The proof uses nonlocal unique continuation and approximation by solutions with exterior data. For fractional semilinear equations with an analytic nonlinearity, \cite{LL2020inverse} proves simultaneous recovery of a cavity and the nonlinearity from the partial exterior DN map for sufficiently small data. A single nonzero sufficiently small exterior datum determines the cavity and the values of the nonlinearity along the corresponding solution.
	
	For semilinear elliptic equations, simultaneous recovery of an unknown cavity and the Taylor coefficients of a semilinear term was proved in \cite{lassas2021partial}. For the equation $\Delta u+a(x,u)=0$, the cavity result assumes $a(x,0)=\partial_u a(x,0)=0$. The first linearized equation is therefore the Laplace equation, and higher-order linearizations determine the remaining Taylor coefficients. Related simultaneous recovery results for semilinear equations with quadratic gradient terms were obtained in \cite[Theorem 1.6]{KU2019partial}. The recovery of cavities and nonlinear coefficients from local boundary data was extended to nonlinearities involving noninteger powers in \cite[Theorem 1.4]{LLST2022inverse}.
	
	Under piecewise analyticity assumptions, \cite[Theorem 2.2]{harrach2022simultaneous} determines a cavity, a scalar diffusion coefficient, and a nonlinear absorption term from local boundary data. For a semilinear Helmholtz equation with a known constant linear term, \cite[Theorem 3.1]{lu2022inverse} determines a cavity and the coefficient of a quadratic nonlinearity from the DN map on the outer boundary. The corresponding simultaneous recovery problem for the linear Schr\"odinger equation is explicitly left open in \cite{lassas2021partial,KU2019partial,LLST2022inverse}.
	
	The proof is inspired in part by the solution to the Calder\'on problem for Schr\"odinger operators on Riemannian surfaces in \cite{carstea_liimatainen_tzou_2024}. We adapt the comparison of exponential CGO growth with polynomial bounds from a Carleman estimate whose harmonic weight need not be Morse. This allows us to determine the obstacle without first recovering the potential.
	
	Finally, let us comment on the announced resolution of the anisotropic Calder\'on problem. Our results do not rely or follow from it. Moreover, its method of proof proceeds in local steps that require numerous or even infinite number of measurements. This contrasts with our main result, which shows that a single measurement suffices to recover the obstacle.

	\para{Main result} We answer the simultaneous recovery question in dimension two. The main difficulty is to determine the obstacle without knowing the surrounding potential. Indeed, for two configurations, the corresponding solutions satisfy $(-\Delta+q_1)(u_1-u_2)=(q_2-q_1)u_2$ in their common exterior, so equality of the outer Cauchy data does not yield a homogeneous equation for their difference. We first prove an obstacle recovery result from a single suitably chosen boundary measurement.
	
	\begin{theorem}[Single measurement recovery of the obstacle]\label{thm:Schiffer}
		Let $\Omega\subset\mathbb R^2$ be a bounded domain with $C^\infty$ boundary, and let $D_j\Subset\Omega$ be admissible obstacles for $j=1,2$. Let $q_j\in C^{1,\alpha}(\overline{\Omega}\setminus D_j)$ be real-valued for $j=1,2$, where $0<\alpha<1$. Assume that $0$ is not a Dirichlet eigenvalue of $-\Delta+q_j$ in $\Omega\setminus\overline{D_j}$ for $j=1,2$, and let $\Lambda_{D_j,q_j}$ be the DN map associated with
		\begin{equation}\label{eq:main_j}
			\begin{cases}
				(-\Delta+q_j)u_j=0&\text{in }\Omega\setminus\overline{D_j},\\
				u_j=0&\text{on }\partial D_j,\\
				u_j=f&\text{on }\partial\Omega.
			\end{cases}
		\end{equation}
		For each fixed pair of configurations $(D_1,q_1)$ and $(D_2,q_2)$, there exists a nonzero $f\in H^{1/2}(\partial\Omega)$, which may depend on both configurations, such that
		\begin{equation}\label{eq:single_Cauchy_data}
			\Lambda_{D_1,q_1}f=\Lambda_{D_2,q_2}f
		\end{equation}
		implies $D_1=D_2$. When $D_1\ne D_2$, we may choose $f$ to be the trace on $\partial\Omega$ of a CGO solution of $(-\Delta+q_1)u=0$ with $u=0$ on $\partial D_1$.
	\end{theorem}
	
	Moreover, the following result answers \ref{Q:IP}.
	
	\begin{corollary}[Simultaneous recovery]\label{cor:simultaneous_recovery}
		Under the assumptions of Theorem \ref{thm:Schiffer}, suppose that
		\begin{equation}\label{eq:same_DN}
			\Lambda_{D_1,q_1}f=\Lambda_{D_2,q_2}f
			\quad\text{for all }f\in H^{1/2}(\partial\Omega).
		\end{equation}
		Then $D_1=D_2=:D$ and $q_1=q_2$ in $\Omega\setminus\overline D$.
	\end{corollary}
	
	The obstacles may have several connected components, and no convexity or analyticity is imposed. Neither potential is required to be known near the outer boundary.
	
	\begin{remark}\label{rem:single_measurement_scope}
		The datum in Theorem \ref{thm:Schiffer} may depend on both configurations. For a fixed pair with $D_1\ne D_2$, the construction in Section \ref{sec:contradiction} gives a family of CGO traces indexed by $h$. Choosing one sufficiently small value of $h$ gives a single Dirichlet datum whose Neumann measurements distinguish that pair. Corollary \ref{cor:simultaneous_recovery} assumes equality of the DN maps for all boundary data and uses this additional information to recover the potential. Since the potentials are real-valued, if a complex datum gives different Neumann measurements, then at least one of its real and imaginary parts does also. Thus the distinguishing datum may be taken real-valued.
	\end{remark}
	
	\para{Outline of the proof} We extend $q_1$ to a larger domain, preserving its $C^{1,\alpha}$ regularity. We extend $q_2$ by using the same function outside $\Omega$ and keeping its original values in $\Omega\setminus\overline{D_2}$. The two extended potentials therefore agree in the added collar. The extension of $q_2$ belongs to $L^\infty$ and may jump across $\partial\Omega$. If the original solutions $u_1$ and $u_2$ have the same Cauchy data on $\partial\Omega$, we extend $u_2$ by $u_1$ in the collar. Equality of the Dirichlet and Neumann traces makes this joined function a weak solution for the extended $q_2$.
	
	Using \cite{guillarmou2011calderon}, we construct a holomorphic Morse phase $\Phi_1=\varphi_1+\mathsf i\psi_1$ with $\psi_1=0$ on $\partial D_1$. The Cauchy--Riemann equations give $\partial_\nu\varphi_1=0$ there. On the enlarged exterior of $D_2$, we choose the harmonic function $\varphi_2$ with the same Dirichlet values as $\varphi_1$ on the enlarged outer boundary and with $\partial_\nu\varphi_2=0$ on $\partial D_2$. A finite-perimeter energy argument proves that $\varphi_1=\varphi_2$ in the added collar implies $D_1=D_2$. This argument applies without a transversality assumption on the obstacle boundaries.
	
	If $D_1\ne D_2$, there are a ball in the added collar and a sign $\sigma\in\{1,-1\}$ such that $\sigma(\varphi_1-\varphi_2)$ has a positive lower bound on that ball. We choose a CGO solution $u_{1,\sigma,h}$ with phase $\sigma\Phi_1$ and use its trace on $\partial\Omega$ as the Dirichlet datum. If the Neumann measurements agreed, the joined solution $u_2$ would belong to $H^2$ on the enlarged exterior of $D_2$. The Carleman estimate with weight $\sigma\varphi_2$ would give a polynomial upper bound for $\|e^{-\sigma\varphi_2/h}u_2\|_{L^2}$. On the chosen ball, the joined solution equals the CGO solution, whose representation gives an exponential lower bound for the same norm. These bounds are incompatible for all sufficiently small $h$. The Carleman estimate allows degenerate critical points of $\varphi_2$.
	
	Once $D_1=D_2=:D$ is known, equality of the outer DN maps gives equality of the partial Cauchy data on $\Omega\setminus\overline D$, with measurements on $\partial\Omega$ and zero Dirichlet values on $\partial D$. Applying \cite[Theorem 1.1]{guillarmou2011calderon} on this fixed domain determines the potential.
	
	\para{Organization of the article} The paper is organized as follows. In Section \ref{sec:prel}, we extend the potentials and construct the harmonic phases and CGO solutions. In Section \ref{sec:Carleman}, we prove the Carleman estimate. We then determine the obstacle in Section \ref{sec:contradiction} and recover the potential in Section \ref{sec:potential}.

	\section{Preliminaries}\label{sec:prel}
	
	Let $D$ be an admissible obstacle and let $q\in L^\infty(\Omega\setminus\overline D)$ be real-valued and satisfy \eqref{eq:eigenvalue_condition}. Set $\mathcal O:=\Omega\setminus\overline D$ and $A_q(u,v):=\int_{\mathcal O}(\nabla u\cdot\nabla\overline v+qu\overline v)\,dx$. The Fredholm alternative gives a unique solution $u_f\in H^1(\mathcal O)$ of \eqref{eq:main} for every $f\in H^{1/2}(\partial\Omega)$, with $\|u_f\|_{H^1(\mathcal O)}\le C\|f\|_{H^{1/2}(\partial\Omega)}$.
	
	For $g\in H^{1/2}(\partial\Omega)$, choose $Eg\in H^1(\mathcal O)$ with trace $g$ on $\partial\Omega$, zero trace on $\partial D$, and $\|Eg\|_{H^1(\mathcal O)}\le C\|g\|_{H^{1/2}(\partial\Omega)}$. Define $\langle\Lambda_{D,q}f,g\rangle:=A_q(u_f,Eg)$, with the pairing conjugate linear in $g$. This is independent of $Eg$, since the difference of two such extensions belongs to $H_0^1(\mathcal O)$ and $u_f$ solves the equation weakly. The preceding bounds give $\|\Lambda_{D,q}f\|_{H^{-1/2}(\partial\Omega)}\le C\|f\|_{H^{1/2}(\partial\Omega)}$. Since $q$ is real-valued, $\Lambda_{D,q}\overline f=\overline{\Lambda_{D,q}f}$, so equality of the DN maps on real data also gives equality on complex data. No spectral assumption is imposed on the enlarged domains below.

	We now work with the two configurations in Theorem \ref{thm:Schiffer}. Choose smooth domains $\widetilde\Omega$ and $\widehat\Omega$ with $\Omega\Subset\widetilde\Omega\Subset\widehat\Omega$ by adjoining sufficiently thin exterior collars to each component of $\partial\Omega$. Set
	\[
	\omega:=\widetilde\Omega\setminus\overline\Omega,
	\quad
	\widetilde\Omega_j:=\widetilde\Omega\setminus\overline{D_j},
	\quad
	\widehat\Omega_1:=\widehat\Omega\setminus\overline{D_1}.
	\]
	Each added collar meets the original connected exterior domain, so these enlarged exterior domains are connected.
	
	Choose a real-valued $C^{1,\alpha}$ extension of $q_1$ to $\overline{\widehat\Omega_1}$ by extending it in smooth boundary charts, and continue to denote this extension by $q_1$. On $\widetilde\Omega_2$, define
	\[
	q_2^{\mathrm{ext}}(x):=
	\begin{cases}
		q_2(x),&x\in\Omega\setminus\overline{D_2},\\
		q_1(x),&x\in\omega.
	\end{cases}
	\]
	The values on $\partial\Omega$ may be assigned arbitrarily. The resulting function belongs to $L^\infty(\widetilde\Omega_2)$ and may have a jump across $\partial\Omega$. In arguments on the enlarged domains, we write $q_2$ for $q_2^{\mathrm{ext}}$. Thus, $q_1=q_2$ in $\omega$ by construction. For $j=1,2$, the extended $q_j$ agrees with the original potential in $\Omega\setminus\overline{D_j}$. This equality in the added collar imposes no agreement between the original potentials near $\partial\Omega$. The extensions are fixed independently of $h$.
	
	The CGO solutions will be constructed in $\widehat\Omega_1$ and restricted to $\widetilde\Omega_1$, so that $\partial\widetilde\Omega$ lies strictly inside their domain of definition.
	
	Choose a nonempty open arc $\Gamma_*\subset\partial\widehat\Omega$ and set $\widehat\Gamma_0:=\partial\widehat\Omega_1\setminus\Gamma_*$. Fix $p\in\widetilde\Omega_1$ and choose $\delta>0$ so that $\overline{B(p,\delta)}\subset\widetilde\Omega_1$. Applying \cite[Proposition 2.1]{guillarmou2011calderon} to $\overline{\widehat\Omega_1}$ gives a holomorphic function $\Phi_1=\varphi_1+\mathsf i\psi_1$ that is Morse up to the boundary, has a nondegenerate critical point $p_0\in B(p,\delta)$, and satisfies $\psi_1=0$ on $\widehat\Gamma_0$.
	
	We also use the boundary regularity supplied by that construction. The function is constructed on an auxiliary smooth surface whose boundary contains $\widehat\Gamma_0$, while $\partial\widehat\Omega_1\setminus\widehat\Gamma_0$ lies in its interior. Its imaginary part vanishes on an open boundary neighborhood of $\widehat\Gamma_0$. Smooth Dirichlet regularity gives smoothness of $\psi_1$ near this boundary set, and the Cauchy--Riemann equations give smoothness of $\varphi_1$ there. Interior regularity applies at the remaining points of $\overline{\widehat\Omega_1}$. Thus, $\Phi_1\in C^\infty(\overline{\widehat\Omega_1})$. Its critical points on this compact set are nondegenerate and hence form a finite set. The Cauchy--Riemann equations show that $\varphi_1$ and $\psi_1$ have the same critical points as $\Phi_1$. At each such point, $\det D^2\varphi_1=-|\Phi_1''|^2<0$ in a complex coordinate. In particular, both $\varphi_1$ and $-\varphi_1$ are harmonic Morse functions.
	
	Since $\partial D_1\subset\widehat\Gamma_0$, one has $\psi_1=0$ and $\partial_\nu\varphi_1=0$ on $\partial D_1$. The trace of $\varphi_1$ on $\partial\widetilde\Omega$ is nonconstant. Indeed, if that trace were a constant $c$, then $w:=\varphi_1-c$ would have zero Dirichlet data on $\partial\widetilde\Omega$ and zero normal derivative on $\partial D_1$. Green's formula would give $\int_{\widetilde\Omega_1}|\nabla w|^2\,dx=0$. Connectedness and the outer Dirichlet condition would imply $\varphi_1=c$ in $\widetilde\Omega_1$. This contradicts the nondegenerate critical point $p_0$ in that domain.
	
	Let $\varphi_2$ solve
	\begin{equation*}
		\begin{cases}
			\Delta\varphi_2=0&\text{in }\widetilde\Omega_2,\\
			\varphi_2=\varphi_1&\text{on }\partial\widetilde\Omega,\\
			\partial_\nu\varphi_2=0&\text{on }\partial D_2.
		\end{cases}
	\end{equation*}
	The mixed problem has a unique real-valued solution $\varphi_2\in C^\infty(\overline{\widetilde\Omega_2})$, since $\widetilde\Omega_2$ is connected, the Dirichlet boundary is nonempty, and the two boundary conditions are imposed on disjoint smooth components. The nonconstant outer value $\varphi_1|_{\partial\widetilde\Omega}$ makes $\varphi_2$ nonconstant.
	
	The CGO construction with boundary correction in \cite[Section 4]{guillarmou2011calderon} applies to both $\Phi_1$ and $-\Phi_1$. We state the representation and remainder estimates needed below.
	
	\begin{lemma}\label{lem:CGO_with_boundary_correction}
		Let $\sigma\in\{1,-1\}$, and let $B\Subset\widetilde\Omega_1$ be a nonempty open set. There exists an amplitude $a_\sigma\in C^\infty(\overline{\widehat\Omega_1})$, holomorphic in $\widehat\Omega_1$, independent of $h$, real-valued on $\partial D_1$, and not identically zero on $B$, such that for all sufficiently small $h>0$ the function
		\begin{equation}\label{eq:CGO_representation_u1}
			u_{1,\sigma,h}=e^{\sigma\Phi_1/h}(a_\sigma+r_{\sigma,h})-e^{\sigma\overline{\Phi_1}/h}(\overline{a_\sigma}+\overline{r_{\sigma,h}})+e^{\sigma\varphi_1/h}\widetilde r_{\sigma,h}
		\end{equation}
		satisfies
		\begin{equation}\label{eq:u1}
			\begin{cases}
				(-\Delta+q_1)u_{1,\sigma,h}=0 &\text{in }\widehat\Omega_1,\\
				u_{1,\sigma,h}=0 &\text{on }\partial D_1.
			\end{cases}
		\end{equation}
		The remainders satisfy $\|r_{\sigma,h}\|_{L^2(\widehat\Omega_1)}=\mathcal O(h)$ and $\|\widetilde r_{\sigma,h}\|_{L^2(\widehat\Omega_1)}=\mathcal O(h^{3/2}|\log h|)$. The amplitude may be chosen nonzero at $p_0$ and to vanish to sufficiently high order at every other critical point of $\Phi_1$ on $\overline{\widehat\Omega_1}$. For each fixed $h$, the restriction of the exact solution $u_{1,\sigma,h}$ belongs to $H^2(\widetilde\Omega_1)$ and has zero trace on $\partial D_1$.
	\end{lemma}
	
	\begin{proof}
		Choose an open boundary neighborhood $\Gamma_0^\sharp$ of $\widehat\Gamma_0$ whose complement contains an open subarc of $\Gamma_*$. The interpolation result in \cite[Lemma 2.3]{guillarmou2011calderon} gives a holomorphic amplitude $a_\sigma$, smooth up to the boundary and real-valued on $\Gamma_0^\sharp$, with $a_\sigma(p_0)\ne0$ and zeros of sufficiently high order at the other critical points. If $a_\sigma$ vanished identically on the nonempty open set $B$, its holomorphicity and the connectedness of $\widehat\Omega_1$ would imply $a_\sigma\equiv0$ throughout $\widehat\Omega_1$, contradicting $a_\sigma(p_0)\ne0$.
		
		Write $L_1:=-\Delta+q_1$. By \cite[Lemma 4.3]{guillarmou2011calderon}, the phase $\sigma\Phi_1$ admits a correction $s_{\sigma,h}\in H^2(\widehat\Omega_1)$ satisfying $\|s_{\sigma,h}\|_{L^2}=\mathcal O(h)$ and $\|e^{-\sigma\Phi_1/h}L_1e^{\sigma\Phi_1/h}(a_\sigma+s_{\sigma,h})\|_{L^2}=\mathcal O(h|\log h|)$. The term supported near $p_0$ vanishes near the boundary, and the remaining term is $h$ times a fixed $H^2$ function. The boundary function $k_\sigma:=h^{-1}s_{\sigma,h}|_{\partial\widehat\Omega_1}$ belongs to $H^{3/2}(\partial\widehat\Omega_1)$ and is independent of $h$, as observed in \cite[proof of Lemma 4.4]{guillarmou2011calderon}.
		
		Choose a holomorphic $b_{0,\sigma}\in H^2(\widehat\Omega_1)$ with $\operatorname{Re}b_{0,\sigma}=-\operatorname{Im}k_\sigma$ on $\Gamma_0^\sharp$, using \cite[Corollary 2.2(ii)]{guillarmou2011calderon}, and put $a_{0,\sigma}:=\mathsf i b_{0,\sigma}$. This function is independent of $h$ and satisfies $\operatorname{Im}a_{0,\sigma}=-\operatorname{Im}k_\sigma$. Set $r_{\sigma,h}:=s_{\sigma,h}+ha_{0,\sigma}$. The boundary value of $a_\sigma+r_{\sigma,h}$ is real on $\widehat\Gamma_0$, and $\|r_{\sigma,h}\|_{L^2}=\mathcal O(h)$. Since $e^{\sigma\Phi_1/h}a_{0,\sigma}$ is holomorphic, $e^{-\sigma\Phi_1/h}L_1e^{\sigma\Phi_1/h}(ha_{0,\sigma})=hq_1a_{0,\sigma}$. The residual remains $\mathcal O(h|\log h|)$ in $L^2$. This is the boundary adjustment in \cite[Section 4.2]{guillarmou2011calderon}, with a real boundary amplitude.
		
		Define $E_{\sigma,h}:=e^{\sigma\Phi_1/h}(a_\sigma+r_{\sigma,h})$ and $W_{\sigma,h}:=E_{\sigma,h}-\overline{E_{\sigma,h}}$. Both $\Phi_1$ and $a_\sigma+r_{\sigma,h}$ are real on $\widehat\Gamma_0$, so $W_{\sigma,h}=0$ there. The real-valued potential gives $L_1\overline{E_{\sigma,h}}=\overline{L_1E_{\sigma,h}}$. The factors $e^{\pm\mathsf i\sigma\psi_1/h}$ have modulus one, and the preceding estimate gives
		\begin{equation}\label{eq:signed_CGO_residual}
			F_{\sigma,h}:=e^{-\sigma\varphi_1/h}L_1W_{\sigma,h},\quad \|F_{\sigma,h}\|_{L^2(\widehat\Omega_1)}\le Ch|\log h|.
		\end{equation}
		The functions $W_{\sigma,h}$ and $F_{\sigma,h}$ are purely imaginary.
		
		Apply \cite[Lemma 4.5]{guillarmou2011calderon} with weight $-\sigma\varphi_1$. Its conjugated operator is $P_{\sigma,h}:=e^{-\sigma\varphi_1/h}L_1e^{\sigma\varphi_1/h}$. Writing $F_{\sigma,h}=\mathsf i f_{\sigma,h}$ with $f_{\sigma,h}$ real, solve $P_{\sigma,h}s=-f_{\sigma,h}$ for a real function $s$ with zero Dirichlet data on $\widehat\Gamma_0$, and set $\widetilde r_{\sigma,h}:=\mathsf i s$. We obtain $P_{\sigma,h}\widetilde r_{\sigma,h}=-F_{\sigma,h}$ and $\|\widetilde r_{\sigma,h}\|_{L^2}\le Ch^{1/2}\|F_{\sigma,h}\|_{L^2}\le Ch^{3/2}|\log h|$. The function $u_{1,\sigma,h}:=W_{\sigma,h}+e^{\sigma\varphi_1/h}\widetilde r_{\sigma,h}$ satisfies \eqref{eq:u1} and \eqref{eq:CGO_representation_u1}.
		
		For fixed $h$, write $u:=u_{1,\sigma,h}\in L^2(\widehat\Omega_1)$. At this stage, the zero Dirichlet condition obtained from \cite[Lemma 4.5]{guillarmou2011calderon} is understood in the transposition sense. Combining its weak identity with the cancellation of the residual of $W_{\sigma,h}$ gives
		\[
		\int_{\widehat\Omega_1}u\,\overline{L_1\eta}\,dx=0
		\]
		for every $\eta\in C^\infty(\overline{\widehat\Omega_1})$ such that $\eta=0$ on $\partial\widehat\Omega_1$ and $\partial_\nu\eta=0$ on $\partial\widehat\Omega$. The normal derivative of $\eta$ is unrestricted on $\partial D_1$; this identity expresses the zero Dirichlet condition there. Taking compactly supported $\eta$ gives $L_1u=0$ in distributions. Since $q_1u\in L^2(\widehat\Omega_1)$, interior elliptic regularity gives $u\in H^2_{\mathrm{loc}}(\widehat\Omega_1)$. We next use the same identity to obtain $H^2$ regularity and the zero trace at $\partial D_1$.
		
		Choose a real-valued $\chi\in C^\infty(\overline{\widehat\Omega_1})$ equal to one near $\partial D_1$ and zero near $\partial\widehat\Omega$. Set $z:=\chi u$ and $b:=[L_1,\chi]u=-2\nabla\chi\cdot\nabla u-(\Delta\chi)u$. The derivatives of $\chi$ are supported in a compact subset of $\widehat\Omega_1$, so $b\in L^2(\widehat\Omega_1)$. Let $\eta\in C^\infty(\overline{\widehat\Omega_1})$ have zero boundary values. Then $\chi\eta$ is an admissible test function in the preceding identity, and
		\begin{equation*}
			\int_{\widehat\Omega_1}z\overline{L_1\eta}\,dx
			=-\int_{\widehat\Omega_1}u\overline{[L_1,\chi]\eta}\,dx
			=\int_{\widehat\Omega_1}b\overline\eta\,dx.
		\end{equation*}
		The last equality follows by integration by parts on the compact support of the derivatives of $\chi$. Both sides are continuous in the $H^2$ norm of $\eta$, so the identity holds for every $\eta\in H^2(\widehat\Omega_1)\cap H_0^1(\widehat\Omega_1)$.
		
		Choose $\lambda>\|q_1\|_{L^\infty(\widehat\Omega_1)}$. The Dirichlet problem $(L_1+\lambda)z_0=b+\lambda z$ has a unique solution $z_0\in H^2(\widehat\Omega_1)\cap H_0^1(\widehat\Omega_1)$. Green's formula and the preceding identity give
		\begin{equation*}
			\int_{\widehat\Omega_1}(z-z_0)\overline{(L_1+\lambda)\eta}\,dx=0,
			\quad \eta\in H^2(\widehat\Omega_1)\cap H_0^1(\widehat\Omega_1).
		\end{equation*}
		Taking $\eta$ to be the Dirichlet solution of $(L_1+\lambda)\eta=z-z_0$ gives $\|z-z_0\|_{L^2(\widehat\Omega_1)}^2=0$. Hence, $z=z_0\in H^2(\widehat\Omega_1)\cap H_0^1(\widehat\Omega_1)$.
		
		Since $u=z$ near $\partial D_1$, the function $u$ belongs to $H^2$ up to this boundary and has zero trace there. Interior regularity elsewhere, together with $\widetilde\Omega\Subset\widehat\Omega$, gives $u\in H^2(\widetilde\Omega_1)$. Only $L_1+\lambda$ was inverted; no spectral assumption on $L_1$ in $\widehat\Omega_1$ is needed.
	\end{proof}
	
	\begin{lemma}\label{lem:gluing_observation}
		Assume that $0$ is not a Dirichlet eigenvalue for either original problem and that $q_1=q_2$ in $\omega$. Let $u_1\in H^1(\widetilde\Omega_1)$ solve $(-\Delta+q_1)u_1=0$ with $u_1=0$ on $\partial D_1$, and set $f:=u_1|_{\partial\Omega}$. Let $u_2$ be the solution of
		\begin{equation}\label{eq:u2}
			\begin{cases}
				(-\Delta+q_2)u_2=0&\text{in }\Omega\setminus\overline{D_2},\\
				u_2=0&\text{on }\partial D_2,\\
				u_2=f&\text{on }\partial\Omega.
			\end{cases}
		\end{equation}
		If $\Lambda_{D_1,q_1}f=\Lambda_{D_2,q_2}f$, extension by $u_2:=u_1$ in $\omega$ gives an $H^1(\widetilde\Omega_2)$ weak solution of $(-\Delta+q_2)u_2=0$.
	\end{lemma}
	
	\begin{proof}
		Uniqueness on $\Omega\setminus\overline{D_1}$ identifies the restriction of $u_1$ with the solution defining $\Lambda_{D_1,q_1}f$. The two Dirichlet traces on $\partial\Omega$ agree, so the piecewise definition of $u_2$ gives an $H^1$ function. For $\eta\in C_c^\infty(\widetilde\Omega_2)$, Green's formula on the two sides of $\partial\Omega$ gives
		\begin{equation*}
			\int_{\widetilde\Omega_2}(\nabla u_2\cdot\nabla\overline\eta+q_2u_2\overline\eta)\,dx=\big\langle\Lambda_{D_2,q_2}f-\Lambda_{D_1,q_1}f,\eta|_{\partial\Omega}\big\rangle=0.
		\end{equation*}
		The normals of the two regions on their common boundary are opposite. In the collar, $u_1$ solves the equation with $q_2$ because the extended potentials agree in $\omega$. This extension is obtained directly from $u_1$ in $\omega$ and requires no additional spectral assumption on $\widetilde\Omega_2$.
	\end{proof}

	\section{Carleman estimate}\label{sec:Carleman}
	
	The comparison function $\varphi_2$ need not be Morse. We prove a Carleman estimate for a nonconstant harmonic weight that may have degenerate critical points. Its additional power of $h^{-1}$ does not affect the contradiction with the exponential lower bound in Section \ref{sec:contradiction}.
	
	Let $\mathcal O$ be a smooth bounded planar domain and write $\partial\mathcal O=\Gamma\cup\Gamma_0$, where $\Gamma$ and $\Gamma_0$ are disjoint unions of boundary components and $\Gamma\ne\emptyset$. In this section, $\nu$ denotes the unit outer normal to $\mathcal O$.
	
	\begin{proposition}[\text{\cite[Proposition 3.1]{guillarmou2011calderon}}]\label{prop:GT11_Carleman}
		Let $\mathcal O$ be a smooth bounded planar domain, and let $\varphi\in C^\infty(\overline{\mathcal O})$ be a real-valued harmonic function that is Morse on $\overline{\mathcal O}$ and satisfies $\partial_\nu\varphi=0$ on $\Gamma_0$. Then for every $V\in L^\infty(\mathcal O)$ there exist constants $C>0$ and $h_0>0$ such that, for all $0<h<h_0$ and all $w\in C^\infty(\overline{\mathcal O})$ with $w|_{\partial\mathcal O}=0$, one has
		\begin{equation}\label{eq:GT11_Carleman}
			\begin{split}
				& h^{-1}\|w\|_{L^2(\mathcal O)}^2
				+ h^{-2}\||d\varphi|w\|_{L^2(\mathcal O)}^2
				+ \|dw\|_{L^2(\mathcal O)}^2
				+ \|\partial_\nu w\|_{L^2(\Gamma_0)}^2 \\
				& \quad \le C\big(
				\|e^{-\varphi/h}(-\Delta+V)e^{\varphi/h}w\|_{L^2(\mathcal O)}^2
				+h^{-1}\|\partial_\nu w\|_{L^2(\Gamma)}^2
				\big).
			\end{split}
		\end{equation}
	\end{proposition}
	
	We construct a convexifying function by solving a Poisson--Neumann problem. The normal derivative on the boundary is estimated in circular conformal coordinates on a collar of each boundary component, as in \cite[Section 3]{guillarmou2011calderon}.
	
	\begin{proposition}\label{prop:harmonic_Carleman}
		Let $\mathcal O$ be a smooth bounded planar domain, and let $\partial\mathcal O=\Gamma\cup\Gamma_0$, where $\Gamma$ and $\Gamma_0$ are disjoint unions of boundary components and $\Gamma\neq\emptyset$. Let $\varphi\in C^\infty(\overline{\mathcal O})$ be a real-valued nonconstant harmonic function in $\mathcal O$ such that $\partial_\nu\varphi=0$ on $\Gamma_0$. Then for every $V\in L^\infty(\mathcal O)$ there exist constants $C>0$ and $h_0>0$ such that, for all $0<h<h_0$ and all $w\in H^2(\mathcal O)\cap H_0^1(\mathcal O)$, one has
		\begin{equation}\label{eq:harmonic_Carleman}
			\begin{split}
				& h^{-1}\|w\|_{L^2(\mathcal O)}^2
				+ h^{-2}\||d\varphi|w\|_{L^2(\mathcal O)}^2
				+ \|dw\|_{L^2(\mathcal O)}^2
				+ \|\partial_\nu w\|_{L^2(\Gamma_0)}^2 \\
				& \quad \le
				Ch^{-1}\big(
				\|e^{-\varphi/h}(-\Delta+V)e^{\varphi/h}w\|_{L^2(\mathcal O)}^2
				+h^{-1}\|\partial_\nu w\|_{L^2(\Gamma)}^2
				\big).
			\end{split}
		\end{equation}
		If $\varphi$ is Morse on $\overline{\mathcal O}$, the factor $h^{-1}$ in front of the parentheses on the right-hand side may be omitted, and one recovers Proposition~\ref{prop:GT11_Carleman}.
	\end{proposition}
	
	\begin{proof}
		We first prove the estimates for $w\in C^\infty(\overline{\mathcal O})$ with zero boundary values. Let $|\mathcal O|$ denote the area of $\mathcal O$, let $|\Gamma|:=\mathcal H^1(\Gamma)$, and define $\beta=0$ on $\Gamma_0$ and $\beta=|\mathcal O|/|\Gamma|$ on $\Gamma$. The function $\beta$ is smooth because $\Gamma$ and $\Gamma_0$ are unions of boundary components. The identity $\int_{\partial\mathcal O}\beta\,dS=|\mathcal O|$ is the compatibility condition for
		\begin{equation}\label{eq:convexifier_Neumann}
			\begin{cases}
				\Delta\rho_0=1&\text{in }\mathcal O,\\
				\partial_\nu\rho_0=0&\text{on }\Gamma_0,\\
				\partial_\nu\rho_0=\beta&\text{on }\Gamma,\\
				\displaystyle\int_{\mathcal O}\rho_0\,dx=0.
			\end{cases}
		\end{equation}
		By standard elliptic theory for the Neumann problem, \eqref{eq:convexifier_Neumann} has a unique real-valued solution $\rho_0\in C^\infty(\overline{\mathcal O})$, independent of $h$ and $\epsilon$. For $0<h\le\epsilon<1$, set $\varphi_{\epsilon,h}:=\varphi-h\rho_0/\epsilon$. Then $-\Delta\varphi_{\epsilon,h}=h/\epsilon$ in $\mathcal O$ and $\partial_\nu\varphi_{\epsilon,h}=0$ on $\Gamma_0$.
		
		Cover $\overline{\mathcal O}$ by finitely many interior neighborhoods and collars of its boundary components. On each boundary collar, choose a conformal coordinate $z=x+\mathsf iy$ which maps the collar onto an annulus $A=\{r_0<|z|<1\}$, with $0<r_0<1$, and the corresponding component of $\partial\mathcal O$ onto $|z|=1$. If that component is initially mapped to the inner circle, compose the map with a holomorphic inversion and rescale so that it is mapped to the unit circle. The other circle $|z|=r_0$ corresponds to a curve inside $\mathcal O$. In these coordinates, the Euclidean metric is $e^{2\lambda(z)}|dz|^2$. Smooth boundary regularity for the conformal map gives a fixed smooth function $\lambda$ up to the boundary, with $e^{2\lambda}$ bounded above and below by positive constants. On an interior neighborhood, use the original Euclidean coordinates and set $\lambda=0$.
		
		Write $\Delta_z:=\partial_x^2+\partial_y^2$ and $\vartheta:=\varphi_{\epsilon,h}/h$ in one such neighborhood $A$. The Laplacian in these coordinates satisfies $\Delta=e^{-2\lambda}\Delta_z$, so $-\Delta_z\vartheta=e^{2\lambda}/\epsilon\ge c_*/\epsilon$ for a fixed $c_*>0$. In the following local calculation, all norms and integrals use Euclidean measure in the $z$ coordinate. Write $\nu_z$ for its exterior unit normal and $\tau_z$ for the rotation of $\nu_z$ by $+\pi/2$. In a boundary collar, the functions considered below vanish near $|z|=r_0$, so only $|z|=1$ contributes to the boundary integrals. In an interior neighborhood, the functions have compact support in $A$.
		
		Set $\partial_z=(\partial_x-\mathsf i\partial_y)/2$ and $\partial_{\bar z}=(\partial_x+\mathsf i\partial_y)/2$. For a smooth complex-valued function $v$ supported in this neighborhood, integration by parts gives
		\begin{equation}\label{eq:carleman_first_order_identity}
			\begin{split}
				4\|e^{-\vartheta}\partial_{\bar z}(e^\vartheta v)\|_{L^2(A)}^2
				&=\|(\partial_x+\mathsf i\vartheta_y)v\|_{L^2(A)}^2+\|(\mathsf i\partial_y+\vartheta_x)v\|_{L^2(A)}^2\\
				&\quad\, +\int_A(-\Delta_z\vartheta)|v|^2\,dx\,dy+\int_{\partial A}\partial_{\nu_z}\vartheta\,|v|^2\,ds-\operatorname{Im}\int_{\partial A}\partial_{\tau_z}v\,\overline v\,ds.
			\end{split}
		\end{equation}
		Write $v=v_R+\mathsf i v_I$. The cross term containing derivatives of the weight is $\int_A\nabla\vartheta\cdot\nabla|v|^2$, which gives the terms with $-\Delta_z\vartheta$ and $\partial_{\nu_z}\vartheta$. The remaining cross term is $2\int_A(\partial_xv_I\partial_yv_R-\partial_xv_R\partial_yv_I)$. The divergence theorem gives $2\int_{\partial A}v_I\partial_{\tau_z}v_R$. In a boundary collar, only the integral over $|z|=1$ remains. Parametrize this circle by $z=e^{\mathsf it}$. Since $v_R(e^{\mathsf it})$ and $v_I(e^{\mathsf it})$ are $2\pi$-periodic and $\partial_{\tau_z}=\partial_t$ there, integration by parts gives $\int_{\partial A}v_R\partial_{\tau_z}v_I\,ds=-\int_{\partial A}v_I\partial_{\tau_z}v_R\,ds$. In an interior neighborhood, both boundary integrals are zero. Expanding the imaginary part of $(\partial_{\tau_z}v)\overline v$, we obtain
		\begin{equation*}
			-\operatorname{Im}\int_{\partial A}(\partial_{\tau_z}v)\overline v\,ds=\int_{\partial A}(v_I\partial_{\tau_z}v_R-v_R\partial_{\tau_z}v_I)\,ds=2\int_{\partial A}v_I\partial_{\tau_z}v_R\,ds.
		\end{equation*}
		
		Let $u\in C^\infty(\overline A)$ be real-valued. In a boundary collar, assume that $u=0$ on $|z|=1$ and that $u$ vanishes near $|z|=r_0$. In an interior neighborhood, assume that $u$ has compact support in $A$. Set $\zeta:=e^{-\vartheta}\partial_z(e^\vartheta u)$. The identity $\Delta_z=4\partial_{\bar z}\partial_z$ gives $e^{-\vartheta}(-\Delta_z)e^\vartheta u=-4e^{-\vartheta}\partial_{\bar z}(e^\vartheta\zeta)$. Its squared $L^2(A)$ norm is $16\|e^{-\vartheta}\partial_{\bar z}(e^\vartheta\zeta)\|_{L^2(A)}^2$. Apply \eqref{eq:carleman_first_order_identity} with $v=\zeta$, discard the two nonnegative squared norms, and multiply the resulting inequality by four. Using $-\Delta_z\vartheta\ge c_*/\epsilon$ gives
		\begin{equation*}
			\|e^{-\vartheta}(-\Delta_z)e^\vartheta u\|_{L^2(A)}^2\ge\frac{4c_*}{\epsilon}\|\zeta\|_{L^2(A)}^2+4\int_{\partial A}\partial_{\nu_z}\vartheta\,|\zeta|^2\,ds-4\operatorname{Im}\int_{\partial A}\partial_{\tau_z}\zeta\,\overline\zeta\,ds.
		\end{equation*}
		Since $u$ is real, $4|\zeta|^2=|\nabla u+u\nabla\vartheta|^2$. The boundary value of $u$ is zero, so a second integration by parts gives
		\begin{equation}\label{eq:carleman_real_energy}
			4\|\zeta\|_{L^2(A)}^2=\|du\|_{L^2(A)}^2+\||d\vartheta|u\|_{L^2(A)}^2+\int_A(-\Delta_z\vartheta)u^2\,dx\,dy.
		\end{equation}
		The last integral is at least $c_*\epsilon^{-1}\|u\|_{L^2(A)}^2$. The bulk terms control $\epsilon^{-2}\|u\|_{L^2}^2$, $\epsilon^{-1}\|du\|_{L^2}^2$, and $\epsilon^{-1}h^{-2}\||d\varphi_{\epsilon,h}|u\|_{L^2}^2$.
		
		On the circle $|z|=1$, which corresponds to the chosen boundary component of $\partial\mathcal O$, write $z=e^{\mathsf it}$ and set $a(t):=(\partial_{\nu_z}u)(e^{\mathsf it})$. Since the annulus lies on the side $|z|<1$, its exterior unit normal is $\nu_z=(\cos t,\sin t)$ and its positively oriented unit tangent is $\tau_z=(-\sin t,\cos t)$. In particular, $\partial_{\tau_z}=\partial_t$ along this circle. The tangential derivative of $u$ is zero there, so $\zeta=\tfrac12e^{-\mathsf it}a(t)$ and $\partial_{\tau_z}\zeta=\tfrac12e^{-\mathsf it}(a'(t)-\mathsf i a(t))$. Since $a$ is real, $-4\operatorname{Im}(\partial_{\tau_z}\zeta\,\overline\zeta)=a(t)^2$. The other boundary contribution is $4\partial_{\nu_z}\vartheta|\zeta|^2=h^{-1}\partial_{\nu_z}\varphi_{\epsilon,h}|\partial_{\nu_z}u|^2$. It is zero on $\Gamma_0$. Its absolute value on $\Gamma$ is bounded by $Ch^{-1}|\partial_{\nu_z}u|^2$, uniformly for $h\le\epsilon$, because $h/\epsilon\le1$ and the functions $\varphi$ and $\rho_0$ are fixed.
		
		The identities $\partial_\nu=e^{-\lambda}\partial_{\nu_z}$ and $dS=e^\lambda ds$, together with the fixed bounds on the conformal factor, convert these estimates to the original Euclidean coordinates. We obtain
		\begin{equation}\label{eq:local_convexified_Carleman}
			\begin{split}
				&\epsilon^{-2}\|u\|_{L^2(\mathcal O)}^2+\epsilon^{-1}\|du\|_{L^2(\mathcal O)}^2+\epsilon^{-1}h^{-2}\||d\varphi_{\epsilon,h}|u\|_{L^2(\mathcal O)}^2+\|\partial_\nu u\|_{L^2(\partial\mathcal O)}^2\\
				&\quad\le C\big(\|e^{-\varphi_{\epsilon,h}/h}(-\Delta)e^{\varphi_{\epsilon,h}/h}u\|_{L^2(\mathcal O)}^2+h^{-1}\|\partial_\nu u\|_{L^2(\Gamma)}^2\big).
			\end{split}
		\end{equation}
		There are no boundary terms on an interior neighborhood. For complex-valued $u$, apply the real-valued estimate separately to its real and imaginary parts and add. At this stage the conjugated Laplacian has real coefficients, so its squared norm is the sum of the two squared norms. The potential is added only after this step.
		
		Choose fixed real-valued smooth functions $\chi_\ell$ subordinate to the coordinate neighborhoods, with $\sum_\ell\chi_\ell^2=1$ on $\overline{\mathcal O}$. For $w=0$ on $\partial\mathcal O$, one has $\partial_\nu(\chi_\ell w)=\chi_\ell\partial_\nu w$ and $\sum_\ell\|d(\chi_\ell w)\|_{L^2}^2=\|dw\|_{L^2}^2+\int_{\mathcal O}\sum_\ell|d\chi_\ell|^2|w|^2\,dx$. The other local terms sum to the corresponding global norms. Set $P_{\epsilon,h}^{(0)}:=e^{-\varphi_{\epsilon,h}/h}(-\Delta)e^{\varphi_{\epsilon,h}/h}$. Its commutator is $[P_{\epsilon,h}^{(0)},\chi_\ell]w=[-\Delta,\chi_\ell]w-2h^{-1}\langle d\varphi_{\epsilon,h},d\chi_\ell\rangle w$, and
		\begin{equation*}
			\sum_\ell\|[P_{\epsilon,h}^{(0)},\chi_\ell]w\|_{L^2}^2\le K\big(\|dw\|_{L^2}^2+h^{-2}\||d\varphi_{\epsilon,h}|w\|_{L^2}^2+\|w\|_{L^2}^2\big),
		\end{equation*}
		where $K$ is independent of $h$ and $\epsilon$. Apply \eqref{eq:local_convexified_Carleman} to each $\chi_\ell w$ and sum over $\ell$. The identities above show that the sum of the left-hand sides is at least $\epsilon^{-2}\|w\|_{L^2(\mathcal O)}^2+\epsilon^{-1}\|dw\|_{L^2(\mathcal O)}^2+\epsilon^{-1}h^{-2}\||d\varphi_{\epsilon,h}|w\|_{L^2(\mathcal O)}^2+\|\partial_\nu w\|_{L^2(\partial\mathcal O)}^2$. On the right, use $P_{\epsilon,h}^{(0)}(\chi_\ell w)=\chi_\ell P_{\epsilon,h}^{(0)}w+[P_{\epsilon,h}^{(0)},\chi_\ell]w$ and $|a+b|^2\le2|a|^2+2|b|^2$. Let $C_0$ include the constant in the local estimate and this factor of two. The commutator terms are bounded by $C_0K(\|dw\|_{L^2(\mathcal O)}^2+h^{-2}\||d\varphi_{\epsilon,h}|w\|_{L^2(\mathcal O)}^2+\|w\|_{L^2(\mathcal O)}^2)$. Choose $\epsilon_0\in(0,1)$ so that $C_0K\epsilon_0\le1/4$ and $C_0K\epsilon_0^2\le1/4$. For $0<h\le\epsilon\le\epsilon_0$, these three terms are absorbed by the corresponding terms on the left. Increasing the constant on the right by a fixed factor gives
		\begin{equation}\label{eq:global_convexified_laplacian}
			\begin{split}
				&\epsilon^{-2}\|w\|_{L^2(\mathcal O)}^2+\epsilon^{-1}\|dw\|_{L^2(\mathcal O)}^2+\epsilon^{-1}h^{-2}\||d\varphi_{\epsilon,h}|w\|_{L^2(\mathcal O)}^2+\|\partial_\nu w\|_{L^2(\partial\mathcal O)}^2\\
				&\quad\le C_1\big(\|P_{\epsilon,h}^{(0)}w\|_{L^2(\mathcal O)}^2+h^{-1}\|\partial_\nu w\|_{L^2(\Gamma)}^2\big),
			\end{split}
		\end{equation}
		where $C_1$ is independent of $\epsilon$, $h$, and $w$ in this range.
		
		Set $P_{\epsilon,h}:=P_{\epsilon,h}^{(0)}+V$. Substituting $\|P_{\epsilon,h}^{(0)}w\|_{L^2(\mathcal O)}^2\le2\|P_{\epsilon,h}w\|_{L^2(\mathcal O)}^2+2\|V\|_{L^\infty(\mathcal O)}^2\|w\|_{L^2(\mathcal O)}^2$ into \eqref{eq:global_convexified_laplacian} introduces the term $2C_1\|V\|_{L^\infty(\mathcal O)}^2\|w\|_{L^2(\mathcal O)}^2$ on the right. Fix $\epsilon\in(0,\epsilon_0]$ small enough that $2C_1\|V\|_{L^\infty(\mathcal O)}^2\le\frac12\epsilon^{-2}$. Subtracting this term leaves at least $\frac12\epsilon^{-2}\|w\|_{L^2(\mathcal O)}^2$ on the left. Multiply the resulting inequality by two. Since $\epsilon<1$, we may replace the coefficients of the three interior terms by one and restrict the boundary norm on the left to $\Gamma_0$. We obtain
		\begin{equation}\label{eq:convexified_Carleman}
			\begin{split}
				&\|w\|_{L^2(\mathcal O)}^2+h^{-2}\||d\varphi_{\epsilon,h}|w\|_{L^2(\mathcal O)}^2+\|dw\|_{L^2(\mathcal O)}^2+\|\partial_\nu w\|_{L^2(\Gamma_0)}^2\\
				&\quad\le C_\epsilon\big(\|P_{\epsilon,h}w\|_{L^2(\mathcal O)}^2+h^{-1}\|\partial_\nu w\|_{L^2(\Gamma)}^2\big),
			\end{split}
		\end{equation}
		for $0<h<h_0\le\epsilon$, where $\epsilon$ is now fixed and $C_\epsilon$ is independent of $h$ and $w$.
		
		Set $\widetilde w:=e^{\rho_0/\epsilon}w$. The multiplier and its inverse have bounded first derivatives for this fixed $\epsilon$. Since $w=0$ on $\p \mathcal{O}$, $\partial_\nu\widetilde w=e^{\rho_0/\epsilon}\partial_\nu w$ on $\p \mathcal{O}$. Direct conjugation gives $P_{\epsilon,h}\widetilde w=e^{\rho_0/\epsilon}e^{-\varphi/h}(-\Delta+V)e^{\varphi/h}w$. Also, $d\varphi=d\varphi_{\epsilon,h}+(h/\epsilon)d\rho_0$, so $h^{-2}\||d\varphi|w\|_{L^2}^2\le C_\epsilon h^{-2}\||d\varphi_{\epsilon,h}|\widetilde w\|_{L^2}^2+C_\epsilon\|\widetilde w\|_{L^2}^2$. Applying \eqref{eq:convexified_Carleman} to $\widetilde w$ and using the corresponding bounds for $w$ and $dw$ gives
		\begin{equation*}
			\begin{split}
				&\|w\|_{L^2(\mathcal O)}^2+h^{-2}\||d\varphi|w\|_{L^2(\mathcal O)}^2+\|dw\|_{L^2(\mathcal O)}^2+\|\partial_\nu w\|_{L^2(\Gamma_0)}^2\\
				&\quad\le C\big(\|e^{-\varphi/h}(-\Delta+V)e^{\varphi/h}w\|_{L^2(\mathcal O)}^2+h^{-1}\|\partial_\nu w\|_{L^2(\Gamma)}^2\big).
			\end{split}
		\end{equation*}
		For $0<h<h_0\le1$, multiplying this estimate by $h^{-1}$ and using $h^{-3}\ge h^{-2}$ and $h^{-1}\ge1$ gives \eqref{eq:harmonic_Carleman}.
		
		If $\varphi$ is Morse, its critical points on $\overline{\mathcal O}$ are finite, and $|d\varphi(x)|\ge c|x-p|$ near each critical point $p$. Let $\eta$ be supported in such a neighborhood and have zero boundary trace. Its zero extension belongs to $H^1(\mathbb R^2)$, including when $p\in\partial\mathcal O$. Integration by parts in $\mathbb R^2$ gives $2\|\eta\|_{L^2}^2=-2\operatorname{Re}\int(x-p)\cdot\nabla\eta\,\overline\eta\,dx$. The Cauchy--Schwarz and Young inequalities give $h^{-1}\|\eta\|_{L^2}^2\le C(\|d\eta\|_{L^2}^2+h^{-2}\||d\varphi|\eta\|_{L^2}^2)$. Away from the critical points, this follows from the positive lower bound for $|d\varphi|$. Summing with a fixed partition of unity introduces only a term $C\|w\|_{L^2}^2$, which is absorbed into $h^{-1}\|w\|_{L^2}^2$ for small $h$. The unconvexified estimate above controls the resulting gradient and weighted terms. It gives the Morse estimate without the additional factor $h^{-1}$ on the right.
		
		Let $w\in H^2(\mathcal O)\cap H_0^1(\mathcal O)$, and choose $w_k\in C^\infty(\overline{\mathcal O})\cap H_0^1(\mathcal O)$ with $w_k\to w$ in $H^2(\mathcal O)$. For fixed $h$, the conjugated operator is continuous from $H^2(\mathcal O)$ to $L^2(\mathcal O)$, since $V\in L^\infty(\mathcal O)$. The normal traces converge in $H^{1/2}(\partial\mathcal O)$ and hence in $L^2(\partial\mathcal O)$. Applying the estimates to $w_k$ and passing to the limit proves both versions for $w$.
	\end{proof}
	
	\begin{remark}\label{rem:non_morse_Carleman}\label{rem:componentwise_boundary_Carleman}
		The conformal coordinates are chosen separately near the boundary components, and the local estimates are combined by a partition of unity. No conformal map of the whole domain to an annulus is needed, so $\mathcal O$ need not be doubly connected. The estimate \eqref{eq:harmonic_Carleman} applies to both $\varphi_2$ and $-\varphi_2$ on $\widetilde\Omega_2$, since they are nonconstant harmonic functions with zero normal derivative on $\partial D_2$. The additional factor $h^{-1}$ on the right changes only the power of $h^{-1}$ in the upper bound and does not affect the contradiction with the exponential lower bound in Section \ref{sec:contradiction}.
	\end{remark}

	\begin{lemma}\label{lem:boundary_lifting}
		Let $D\Subset\Omega\Subset\widetilde\Omega$, where $D$ and $\widetilde\Omega$ have smooth boundaries and $\operatorname{dist}(\partial D,\partial\widetilde\Omega)>0$. Let $v\in H^2(\widetilde\Omega\setminus\overline D)$ satisfy $v=0$ on $\partial D$, and assume that its outer traces $v|_{\partial\widetilde\Omega}$ and $\partial_\nu v|_{\partial\widetilde\Omega}$ belong to $W^{2,\infty}(\partial\widetilde\Omega)$. Then there exists $g\in W^{2,\infty}(\widetilde\Omega\setminus\overline D)$, supported in a fixed collar of $\partial\widetilde\Omega$, such that $g=v$ and $\partial_\nu g=\partial_\nu v$ on $\partial\widetilde\Omega$, and
		\begin{equation}\label{eq:boundary_lifting_estimate}
			\|g\|_{W^{2,\infty}(\widetilde\Omega\setminus\overline D)}
			\le C\big(
			\|v\|_{W^{2,\infty}(\partial\widetilde\Omega)}
			+\|\partial_\nu v\|_{W^{2,\infty}(\partial\widetilde\Omega)}
			\big),
		\end{equation}
		where $C>0$ depends only on $\widetilde\Omega$ and the chosen collar.
	\end{lemma}
	
	\begin{proof}
		Let $\kappa(s,y'):=y'-s\nu(y')$ be the inward normal coordinate map in a fixed collar of $\partial\widetilde\Omega$ disjoint from $\overline D$. Put $a:=v|_{\partial\widetilde\Omega}$ and $b:=\partial_\nu v|_{\partial\widetilde\Omega}$. Choose $\eta\in C_c^\infty([0,\delta))$ equal to one near zero, and define $g(\kappa(s,y')):=(a(y')-sb(y'))\eta(s)$ in the collar and $g=0$ outside it. Since $\eta$ vanishes for $s$ near $\delta$, this definition gives $g\in W^{2,\infty}(\widetilde\Omega\setminus\overline D)$. The identity $\partial_\nu=-\partial_s$ at $s=0$ gives the required traces $a$ and $b$.
		
		The boundary $W^{2,\infty}$ norms use tangential derivatives in fixed smooth charts. Tangential derivatives of $g$ of order at most two involve only the corresponding derivatives of $a$ and $b$. Its normal derivatives are $\partial_sg=-b\eta+(a-sb)\eta'$ and $\partial_s^2g=-2b\eta'+(a-sb)\eta''$, and the mixed derivatives involve at most one tangential derivative of $a$ and $b$. The collar map and its inverse have fixed bounded derivatives of the required orders. These formulas give \eqref{eq:boundary_lifting_estimate}.
	\end{proof}
	
	\begin{proposition}\label{prop:Carleman_boundary_terms}
		Let $D\Subset\Omega\Subset\widetilde\Omega$, where $D$ and $\widetilde\Omega$ have smooth boundaries and $\widetilde\Omega\setminus\overline D$ is connected. Let $\varphi\in C^\infty(\overline{\widetilde\Omega}\setminus D)$ be a real-valued nonconstant harmonic function satisfying $\partial_\nu\varphi=0$ on $\partial D$. For every $V\in L^\infty(\widetilde\Omega\setminus\overline D)$ there exist $C>0$ and $h_0>0$ such that, for all $0<h<h_0$ and all $v\in H^2(\widetilde\Omega\setminus\overline D)$ with $v=0$ on $\partial D$ and outer traces $v|_{\partial\widetilde\Omega},\partial_\nu v|_{\partial\widetilde\Omega}\in W^{2,\infty}(\partial\widetilde\Omega)$, the following estimate holds with $N=5$:
		\begin{equation}\label{eq:Carleman_boundary_terms}
			\begin{split}
				&h^{-1}\|v\|_{L^2(\widetilde\Omega\setminus\overline D)}^2
				+h^{-2}\||d\varphi|v\|_{L^2(\widetilde\Omega\setminus\overline D)}^2			+\|dv\|_{L^2(\widetilde\Omega\setminus\overline D)}^2
				+\|\partial_\nu v\|_{L^2(\partial D)}^2\\
				&\le Ch^{-N}\big(
				\|e^{-\varphi/h}(-\Delta+V)e^{\varphi/h}v\|_{L^2(\widetilde\Omega\setminus\overline D)}^2
				+\mathcal B_{\partial\widetilde\Omega}(v)
				\big),
			\end{split}
		\end{equation}
		where
		\begin{equation}\label{eq:def_B_partial_Omega}
			\mathcal B_{\partial\widetilde\Omega}(v)
			:=\|v\|_{W^{2,\infty}(\partial\widetilde\Omega)}^2
			+\|\partial_\nu v\|_{W^{2,\infty}(\partial\widetilde\Omega)}^2.
		\end{equation}
	\end{proposition}
	
	\begin{proof}
		Set $\mathcal O:=\widetilde\Omega\setminus\overline D$, $\Gamma_0:=\partial D$, and $\Gamma:=\partial\widetilde\Omega$. Let $g$ be the function constructed in Lemma \ref{lem:boundary_lifting}, and set $w:=v-g\in H^2(\mathcal O)\cap H_0^1(\mathcal O)$. Then $w=0$ on $\partial\mathcal O$, $\partial_\nu w=\partial_\nu v$ on $\partial D$, and $\partial_\nu w=0$ on $\partial\widetilde\Omega$. Proposition \ref{prop:harmonic_Carleman} gives
		\begin{equation}\label{eq:Carleman_w_new}
			\begin{split}
				h^{-1}\|w\|_{L^2(\mathcal O)}^2+h^{-2}\||d\varphi|w\|_{L^2(\mathcal O)}^2
				+\|dw\|_{L^2(\mathcal O)}^2+\|\partial_\nu v\|_{L^2(\partial D)}^2\le Ch^{-1}\|P_{\varphi,h}w\|_{L^2(\mathcal O)}^2,
			\end{split}
		\end{equation}
		where $P_{\varphi,h}:=e^{-\varphi/h}(-\Delta+V)e^{\varphi/h}$. The term on $\partial\widetilde\Omega$ vanishes because $\partial_\nu w=0$ there. Since $w=v-g$,
		\begin{equation}\label{eq:Pw_split_new}
			\|P_{\varphi,h}w\|_{L^2(\mathcal O)}^2
			\le2\|P_{\varphi,h}v\|_{L^2(\mathcal O)}^2
			+2\|P_{\varphi,h}g\|_{L^2(\mathcal O)}^2.
		\end{equation}
		Since $\Delta\varphi=0$, we have
		\[
		P_{\varphi,h}g
		=-\Delta g-2h^{-1}\langle d\varphi,dg\rangle
		-h^{-2}|d\varphi|^2g+Vg.
		\]
		For $0<h\le1$, this gives $|P_{\varphi,h}g|\le Ch^{-2}\|g\|_{W^{2,\infty}(\mathcal O)}$ almost everywhere. Lemma \ref{lem:boundary_lifting} gives
		\begin{equation}\label{eq:Pg_est_new}
			\|P_{\varphi,h}g\|_{L^2(\mathcal O)}^2
			\le Ch^{-4}\mathcal B_{\partial\widetilde\Omega}(v).
		\end{equation}
		It also gives
		\[
		h^{-1}\|g\|_{L^2}^2
		+h^{-2}\||d\varphi|g\|_{L^2}^2
		+\|dg\|_{L^2}^2
		\le Ch^{-2}\mathcal B_{\partial\widetilde\Omega}(v).
		\]
		Apply $|a+b|^2\le2|a|^2+2|b|^2$ to $v=w+g$, its gradient, and its product with $|d\varphi|$. The inner normal derivative is unchanged because $g$ vanishes near $\partial D$. Combining these bounds with \eqref{eq:Carleman_w_new}--\eqref{eq:Pg_est_new} gives
		\begin{equation*}
			\begin{split}
				h^{-1}\|v\|_{L^2(\mathcal O)}^2
				+h^{-2}\||d\varphi|v\|_{L^2(\mathcal O)}^2
				+\|dv\|_{L^2(\mathcal O)}^2
				+\|\partial_\nu v\|_{L^2(\partial D)}^2\le Ch^{-1}\|P_{\varphi,h}v\|_{L^2(\mathcal O)}^2
				+Ch^{-5}\mathcal B_{\partial\widetilde\Omega}(v).
			\end{split}
		\end{equation*}
		Since $h^{-1}\le h^{-5}$ for $0<h\le1$, this proves \eqref{eq:Carleman_boundary_terms} with $N=5$.
	\end{proof}

	\section{Unique determination of the obstacle}\label{sec:contradiction}
	
	In this section, we use the extensions and phases constructed in Section \ref{sec:prel}. The extended potentials satisfy $q_1\in C^{1,\alpha}(\overline{\widehat\Omega_1})$, $q_2\in L^\infty(\widetilde\Omega_2)$, and $q_1=q_2$ in $\omega$. The domains $\widetilde\Omega_j$ are connected. The functions $\varphi_j$ are nonconstant harmonic, have the same nonconstant trace on $\partial\widetilde\Omega$, and satisfy $\partial_\nu\varphi_j=0$ on $\partial D_j$. Only $\varphi_1$ is required to be Morse; $\varphi_2$ is the Carleman weight for the second configuration.
	
	\begin{theorem}\label{thm:obstacle_uniqueness}
		Under the assumptions of Theorem \ref{thm:Schiffer}, suppose that $D_1\ne D_2$, and use the extensions constructed in Section \ref{sec:prel}. There exist $\sigma\in\{1,-1\}$, an amplitude in Lemma \ref{lem:CGO_with_boundary_correction}, and $h_*>0$ such that, for every $0<h<h_*$, the CGO solution $u_{1,\sigma,h}$ has a nonzero trace $f_h:=u_{1,\sigma,h}|_{\partial\Omega}$ satisfying
		\[
		\Lambda_{D_1,q_1}f_h\ne\Lambda_{D_2,q_2}f_h.
		\]
	\end{theorem}
	
	We first prove that equality of the auxiliary harmonic functions in $\omega$ implies equality of the obstacles. When the obstacles differ, the sign of their difference determines which CGO phase to use. We then compare the weighted solution for the second configuration with this CGO solution.
	
	\subsection{Recovery of the obstacle from phase equality}
	
	Let
	\[
	F:=\widetilde\Omega\setminus\overline{(D_1\cup D_2)},
	\]
	and let $G$ be the union of all connected components of $F$ that meet $\omega$. Hence, $\omega\subset G$, and every connected component of $G$ contains a nonempty open part of $\omega$. This definition includes the collars of all components of $\partial\Omega$. A sufficiently thin two-sided collar of $\partial\Omega$ is disjoint from both obstacles, and each of its components meets $\omega$, so the whole collar lies in $G$.
	
	For $j=1,2$, define
	\[
	D_j^*:=(\widetilde\Omega\setminus\overline G)
	\cap(\widetilde\Omega\setminus\overline{D_j}).
	\]
	The set $D_j^*$ consists of points in the exterior of $D_j$ that lie outside $\overline G$. It is open. Since $\omega$ and the two-sided collar of $\partial\Omega$ lie in $G$, one has $\overline{D_j^*}\Subset\Omega$. We will apply an energy identity on these sets. Their boundaries need not be smooth, and the following argument does not require $F$ or $G$ to have finitely many connected components.
	
	\begin{lemma}\label{lem:difference_nonempty}
		If $D_1\neq D_2$, then at least one of $D_1\setminus\overline{D_2}$ and $D_2\setminus\overline{D_1}$ is nonempty.
	\end{lemma}
	
	\begin{proof}
		If both differences were empty, then $D_1\subset\overline{D_2}$ and $D_2\subset\overline{D_1}$, giving $\overline{D_1}=\overline{D_2}$. Every open set with smooth boundary satisfies $D_j=\operatorname{int}\overline{D_j}$, so equality of the closures gives $D_1=D_2$.
	\end{proof}

	\begin{lemma}\label{lem:trapped_sets_nonempty}
		If $D_1\neq D_2$, then at least one of $D_1^*$ and $D_2^*$ has a nonempty interior.
	\end{lemma}
	
	\begin{proof}
		By Lemma \ref{lem:difference_nonempty}, at least one of $D_2\setminus\overline{D_1}$ and $D_1\setminus\overline{D_2}$ is nonempty. Suppose $D_2\setminus\overline{D_1}\ne\emptyset$. Since $G\cap D_2=\emptyset$ and $D_2$ is open, we also have $\overline G\cap D_2=\emptyset$: any point in that intersection would have a ball contained in $D_2$ that meets $G$. It follows that $D_2\setminus\overline{D_1}\subset D_1^*$. The set on the left is open and nonempty. The other case gives $D_1\setminus\overline{D_2}\subset D_2^*$ by the same argument.
	\end{proof}
	
	We use $\mathcal L^2$ for Lebesgue measure on $\R^2$ and $\mathcal H^1$ for one-dimensional Hausdorff measure. If a measurable set $A\subset\mathbb R^2$ has locally finite perimeter, its characteristic function has a distributional derivative $D\chi_A$ which is a vector-valued Radon measure. The reduced boundary $\partial^*A$ consists of those $x\in\operatorname{supp}|D\chi_A|$ for which $\nu_A(x):=-\lim_{r\downarrow0}\frac{D\chi_A(B(x,r))}{|D\chi_A|(B(x,r))}$ exists and has length one. At these points, the rescaled sets converge locally in measure to the half-plane $\{y:y\cdot\nu_A(x)<0\}$. The vector $\nu_A$ is the measure-theoretic exterior normal, and $D\chi_A=-\nu_A\mathcal H^1\!\restriction_{\partial^*A}$, with $|D\chi_A|=\mathcal H^1\!\restriction_{\partial^*A}$. We use \cite[Definitions 3.35 and 3.54, Theorem 3.59]{ambrosio2000functions} for finite perimeter, reduced boundaries, and the structure theorem, and \cite[Theorem 3.36]{ambrosio2000functions} for Gauss--Green. Our exterior normal is the negative of the inner normal used in that reference.
	
	\begin{lemma}\label{lem:trapped_sets_finite_perimeter}
		The sets $D_1^*$ and $D_2^*$ are bounded sets of finite perimeter. For $j=1,2$, one has $\mathcal H^1(\partial^*D_j^*\setminus(\partial D_j\cup\partial G))=0$ and $\mathcal H^1(\partial^*D_j^*\setminus(\partial D_1\cup\partial D_2))=0$.
	\end{lemma}
	
	\begin{proof}
		The smooth sets $D_1$, $D_2$, and $\widetilde\Omega$ have finite perimeter, and their boundaries have zero area. The set $F$ agrees almost everywhere with $\widetilde\Omega\cap D_1^c\cap D_2^c$ and has finite perimeter. Since $F$ is open, its connected components are open. Both $G$ and $F\setminus G$ are unions of these components, so $G$ is also closed relative to $F$. Every point of $\partial G$ belongs to $\overline F\setminus F$. In particular, $\partial G\subset\partial F\subset\Sigma$, where $\Sigma:=\partial\widetilde\Omega\cup\partial D_1\cup\partial D_2$ is a finite union of smooth compact curves.
		
		We verify the finite perimeter of $G$ without counting its components. Let $\eta_\delta$ be a standard mollifier supported in $B(0,\delta)$ and put $g_\delta:=\chi_G*\eta_\delta$. If $\operatorname{dist}(x,\partial G)>\delta$, the characteristic function is constant on $B(x,\delta)$, so $\nabla g_\delta(x)=0$. Also, $\|\nabla g_\delta\|_{L^\infty}\le C\delta^{-1}$. The fixed smooth curves forming $\Sigma$ have a $\delta$-neighborhood of area at most $C\delta$. It follows that $\int_{\mathbb R^2}|\nabla g_\delta|\,dx\le C$. Since $g_\delta\to\chi_G$ in $L^1$, passage to the limit in $\int g_\delta\operatorname{div}X$ for $X\in C_c^1(\mathbb R^2;\mathbb R^2)$ with $|X|\le1$ proves $|D\chi_G|(\mathbb R^2)\le C$. Thus $G$ has finite perimeter. The inclusion $\partial G\subset\Sigma$ also gives $\mathcal L^2(\partial G)=0$.
		
		Set $\widehat D_j:=\widetilde\Omega\cap G^c\cap D_j^c$, where the complements are taken in $\mathbb R^2$. From the definition of $D_j^*$, one has $D_j^*\subset\widehat D_j$ and $\widehat D_j\setminus D_j^*\subset\partial G\cup\partial D_j$. We proved that $\mathcal L^2(\partial G)=0$, and the smooth boundary $\partial D_j$ also has zero area. The characteristic functions of $D_j^*$ and $\widehat D_j$ agree almost everywhere. Finite-perimeter intersection calculus gives finite perimeter of the bounded set $\widehat D_j$. Since the two characteristic functions are the same element of $L^1(\mathbb R^2)$, their distributional derivatives and total variation measures agree. The definition of the reduced boundary then identifies their reduced boundaries and measure-theoretic normals, up to an $\mathcal H^1$-null set. This proves that $D_j^*$ has finite perimeter.
		
		For finite-perimeter sets $A$ and $B$, the intersection formula in \cite[Theorem 16.3]{maggi2012sets} gives
		\begin{equation}\label{eq:reduced_boundary_intersection_simple}
			\mathcal H^1\big(\partial^*(A\cap B)\setminus(\partial^*A\cup\partial^*B)\big)=0.
		\end{equation}
		Apply \eqref{eq:reduced_boundary_intersection_simple} first to $\widetilde\Omega\cap G^c$ and then to its intersection with $D_j^c$. This gives $\partial^*\widehat D_j\subset\partial^*\widetilde\Omega\cup\partial^*(G^c)\cup\partial^*(D_j^c)$ up to an $\mathcal H^1$-null set. For a finite-perimeter set and its complement, the distributional derivatives have opposite signs, and the total variation measures agree. Their reduced boundaries coincide, with opposite normals. The smooth sets $\widetilde\Omega$ and $D_j$ have their classical boundaries as reduced boundaries, and $\partial^*G\subset\partial G$ because $G$ is open. Using the almost-everywhere identification of $D_j^*$ with $\widehat D_j$, we obtain $\partial^*D_j^*\subset\partial\widetilde\Omega\cup\partial G\cup\partial D_j$ up to an $\mathcal H^1$-null set.
		
		To exclude the outer boundary, choose a sufficiently thin collar of $\partial\widetilde\Omega$ inside $\widetilde\Omega$. Since $\Omega\Subset\widetilde\Omega$, this collar lies in $\omega\subset G$ and is disjoint from $D_j^*$. Outside $\widetilde\Omega$, the characteristic function of $D_j^*$ is zero by definition. It is zero on a full neighborhood of $\partial\widetilde\Omega$, so its distributional derivative has no support there. The preceding inclusion reduces to $\partial^*D_j^*\subset\partial G\cup\partial D_j$ up to an $\mathcal H^1$-null set. Finally, $\partial G\subset\partial\widetilde\Omega\cup\partial D_1\cup\partial D_2$, and the outer boundary has already been excluded. This proves both assertions.
	\end{proof}
	
	For a finite-perimeter set $\mathcal T$, the normal derivative used below is
	\[
	\partial_{\nu_{\mathcal T}}\varphi :=\nabla\varphi_{\mathrm{tr}}\cdot\nu_{\mathcal T}
	\quad\text{on }\partial^*\mathcal T,
	\]
	where $\nu_{\mathcal T}$ is the measure-theoretic exterior normal. When a boundary point lies on $\partial D_j$, the gradient trace of $\varphi_j$ is taken from $\widetilde\Omega_j$. Since $\varphi_j$ is smooth up to $\partial D_j$, it admits a $C^2$ extension across that boundary with the same gradient trace. This extension allows us to apply Gauss--Green on $D_j^*$.
	
	\begin{lemma}\label{lem:neumann_finite_perimeter}
		Let $\mathcal T\subset\mathbb R^2$ be a bounded set of finite perimeter. Let $U\subset\mathbb R^2$ be an open neighborhood of $\overline{\mathcal T}$, and let $\varphi\in C^2(U)$ be real-valued and satisfy $\Delta\varphi=0$ for $\mathcal L^2$-almost every point of $\mathcal T$. Assume that $\nabla\varphi\cdot\nu_{\mathcal T}=0$ for $\mathcal H^1$-a.e. point of $\partial^*\mathcal T$. Then
		\begin{equation}\label{eq:integ_nabla_varphi=0}
			\int_{\mathcal T}|\nabla\varphi|^2\,dx=0.
		\end{equation}
	\end{lemma}
	
	\begin{proof}
		Choose $\chi\in C_c^\infty(U)$ equal to one on a neighborhood of the compact set $\overline{\mathcal T}$. The vector field $X:=\chi\varphi\nabla\varphi$, extended by zero outside $U$, belongs to $C_c^1(\mathbb R^2;\mathbb R^2)$. Since $\Delta\varphi=0$ almost everywhere in $\mathcal T$, the Gauss--Green formula gives
		\begin{equation*}
			\int_{\mathcal T}|\nabla\varphi|^2\,dx=\int_{\mathcal T}\operatorname{div}X\,dx=\int_{\partial^*\mathcal T}\varphi\nabla\varphi\cdot\nu_{\mathcal T}\,d\mathcal H^1=0.
		\end{equation*}
	\end{proof}
	
	\begin{lemma}\label{lem:phase_equality_in_G}
		If $\varphi_1=\varphi_2$ in $\omega$, then $\varphi_1=\varphi_2$ in $G$.
	\end{lemma}
	
	\begin{proof}
		The difference $w:=\varphi_1-\varphi_2$ is harmonic in $G$. Every connected component $G_\alpha$ meets $\omega$ in a nonempty open set where $w=0$. Unique continuation gives $w=0$ in each $G_\alpha$.
	\end{proof}
	
	\begin{lemma}\label{lem:trapped_sets_zero_trace}
		Assume that $\varphi_1=\varphi_2$ in $\omega$. Then
		\[
		\partial_{\nu_{D_j^*}}\varphi_j=0
		\quad\text{for }\mathcal H^1\text{-almost every point of }\partial^*D_j^* \text{ and }j=1,2,
		\]
		with the gradient trace taken as described above.
	\end{lemma}
	
	\begin{proof}
		By Lemma \ref{lem:phase_equality_in_G}, $\varphi_1=\varphi_2$ in $G$. Set $\mathcal T:=D_1^*$. Lemma \ref{lem:trapped_sets_finite_perimeter} reduces the proof to the parts of $\partial^*\mathcal T$ in $\partial D_1$ and in $\partial G\setminus\partial D_1$, up to an $\mathcal H^1$-null set.
		
		Let $n_1$ be the exterior normal to $D_1$. At a reduced boundary point $x\in\partial^*\mathcal T\cap\partial D_1$, the blow-up of $\mathcal T$ is the half-plane $\{y:y\cdot\nu_{\mathcal T}(x)<0\}$. The smooth exterior of $D_1$ has blow-up $\{y:y\cdot n_1(x)>0\}$. The inclusion $\mathcal T\subset\widetilde\Omega\setminus\overline{D_1}$ passes to these limits in local measure. Two half-planes through the origin, one contained in the other up to a null set, are equal. It follows that $\nu_{\mathcal T}(x)=-n_1(x)$ for almost every point under consideration. The exterior trace of $\varphi_1$ satisfies $\partial_{n_1}\varphi_1=0$, so its scalar product with $\nu_{\mathcal T}$ is zero. This argument includes the points of $\partial D_1\cap\partial D_2$; the intersection is not assumed to have zero length.
		
		Now take $x\in\partial^*\mathcal T\cap(\partial G\setminus\partial D_1)$ outside the exceptional null set. This point lies on $\partial D_2$ and is separated from $\partial\widetilde\Omega$. Since $\overline G\cap D_1=\emptyset$ and $x\notin\partial D_1$, one has $x\notin\overline{D_1}$. Choose a boundary neighborhood $W\Subset\widetilde\Omega$ disjoint from $\overline{D_1}$, small enough that $W\setminus\overline{D_2}$ is connected. This set equals $F\cap W$ and meets $G$, because $x\in\partial G$. As a connected subset of $F$, it lies in a single component of $F$, and that component belongs to $G$. We have $G\cap W=W\setminus\overline{D_2}$ and $\mathcal T\cap W=D_2\cap W$.
		
		On $W\setminus\overline{D_2}$, the functions $\varphi_1$ and $\varphi_2$ agree. The first is smooth across $\partial D_2$ in $W$, while the second is smooth up to that boundary from the exterior. Their gradient traces agree there. With $n_2$ the exterior normal to $D_2$, this gives $\nabla\varphi_1(x)\cdot n_2(x)=\partial_{n_2}\varphi_2(x)=0$. The local identity $\mathcal T\cap W=D_2\cap W$ gives $\nu_{\mathcal T}=n_2$. Both boundary parts have zero normal trace, proving the assertion for $D_1^*$. Interchanging the indices proves it for $D_2^*$.
	\end{proof}
	
	\begin{lemma}\label{lem:geometric_uniqueness_phase}
		If $\varphi_1=\varphi_2$ in $\omega$, then $D_1=D_2$.
	\end{lemma}
	
	\begin{proof}
		Suppose $D_1\ne D_2$. Lemma \ref{lem:trapped_sets_nonempty} gives an index $j$ for which the open set $D_j^*$ is nonempty. The function $\varphi_j$ is harmonic in $\widetilde\Omega_j$ and smooth up to $\partial D_j$ from the exterior. Extend it as a $C^2$ function across $\partial D_j$. Since $\overline{D_j^*}\Subset\Omega$, this gives a $C^2$ function on a neighborhood of $\overline{D_j^*}$, unchanged on $D_j^*$ and with the same gradient traces. The extension is not required to be harmonic inside $D_j$; it remains harmonic on $D_j^*$.
		
		Lemma \ref{lem:trapped_sets_zero_trace} gives the zero normal trace required by Lemma \ref{lem:neumann_finite_perimeter}. Applying that lemma gives $\int_{D_j^*}|\nabla\varphi_j|^2\,dx=0$. Continuity of the gradient and openness of $D_j^*$ imply that $\varphi_j$ is constant on a ball contained in $D_j^*$. The domain $\widetilde\Omega_j$ is connected by the collar construction. Unique continuation applied to $\varphi_j$ minus this constant makes it constant throughout $\widetilde\Omega_j$, contradicting its nonconstant trace on $\partial\widetilde\Omega$.
	\end{proof}

	\subsection{Choice of the CGO solution and the polynomial upper bound}
	
	For distinct obstacles, Lemma \ref{lem:geometric_uniqueness_phase} gives a nonzero difference of the comparison functions in $\omega$.
	
	\begin{lemma}\label{lem:phase_gap}
		If $D_1\neq D_2$, then there exist a sign $\sigma\in\{1,-1\}$, a ball $B\Subset\omega$, and a constant $c>0$ such that $d\psi_1\neq0$ on $\overline B$ and
		\begin{equation}\label{eq:varphi_1-varphi_2>c}
			\sigma(\varphi_1-\varphi_2)\ge 2c
			\quad\text{in } B.
		\end{equation}
	\end{lemma}
	
	\begin{proof}
		Lemma \ref{lem:geometric_uniqueness_phase} implies that $\varphi_1-\varphi_2$ does not vanish identically in $\omega$. Choose its sign $\sigma$ so that $U:=\{x\in\omega:\sigma(\varphi_1-\varphi_2)(x)>0\}$ is nonempty. By the Cauchy--Riemann equations, the zeros of $d\psi_1$ are the critical points of $\Phi_1$. They are isolated, so the open set $U$ contains a point $x_1$ with $d\psi_1(x_1)\ne0$. Continuity gives a ball $B$ centered at $x_1$, with $\overline B\subset U$, on which $|d\psi_1|$ has a positive lower bound. Reducing the radius if necessary gives a constant $c>0$ for which \eqref{eq:varphi_1-varphi_2>c} holds.
	\end{proof}
	
	Fix $\sigma$, $B$, and $c$ from Lemma \ref{lem:phase_gap}. Choose the amplitude $a_\sigma$ in Lemma \ref{lem:CGO_with_boundary_correction}, independent of $h$ and not identically zero on $B$, and use the resulting family
	\begin{equation}\label{eq:signed_CGO_section4}
		u_{1,\sigma,h}=e^{\sigma\Phi_1/h}(a_\sigma+r_{\sigma,h})-e^{\sigma\overline{\Phi_1}/h}(\overline{a_\sigma}+\overline{r_{\sigma,h}})+e^{\sigma\varphi_1/h}\widetilde r_{\sigma,h}.
	\end{equation}
	Set $f_h:=u_{1,\sigma,h}|_{\partial\Omega}$. The configurations, phases, sign, and amplitude are fixed from now on. Constants may depend on these choices but are independent of $h$.
	
	\begin{lemma}\label{lem:collar_gluing}
		Assume that $\Lambda_{D_1,q_1}f_h=\Lambda_{D_2,q_2}f_h$, and let $u_2$ solve \eqref{eq:u2} with $f=f_h$. Extension by $u_2:=u_{1,\sigma,h}$ in $\omega$ gives an $H^2(\widetilde\Omega_2)$ weak solution of $(-\Delta+q_2)u_2=0$.
	\end{lemma}
	
	\begin{proof}
		Lemma \ref{lem:CGO_with_boundary_correction} gives $u_{1,\sigma,h}\in H^2(\widetilde\Omega_1)$ with zero trace on $\partial D_1$. Its trace on $\partial\Omega$ is $f_h$, and the extended potentials agree in $\omega$. Hence, Lemma \ref{lem:gluing_observation} gives an $H^1(\widetilde\Omega_2)$ weak solution.
		
		The joined solution has zero Dirichlet trace on $\partial D_2$ and the trace of $u_{1,\sigma,h}$ on $\partial\widetilde\Omega$. The latter belongs to $H^{3/2}(\partial\widetilde\Omega)$ because this boundary lies in the interior of $\widehat\Omega_1$, where $u_{1,\sigma,h}$ is $H^2$. Choose $E\in H^2(\widetilde\Omega_2)$ with this outer trace and vanishing near $\partial D_2$. Then $w:=u_2-E\in H_0^1(\widetilde\Omega_2)$ satisfies
		\[
		-\Delta w=-q_2u_2+\Delta E\in L^2(\widetilde\Omega_2).
		\]
		Dirichlet regularity for the Laplacian gives $w\in H^2(\widetilde\Omega_2)$ and hence $u_2\in H^2(\widetilde\Omega_2)$. This step uses only the boundedness of the extended $q_2$ and does not require a spectral assumption on the enlarged Schr\"odinger operator.
	\end{proof}

	Assuming equality of the Neumann data for $f_h$, we estimate the weighted solution for the second configuration. We first bound the outer boundary values of the weighted CGO solution.
	
	\begin{lemma}\label{lem:outer_boundary_polynomial_control}
		On the interior side of a collar of $\partial\widetilde\Omega$, set $v_1:=e^{-\sigma\varphi_2/h}u_{1,\sigma,h}$. There exist $C>0$ and $h_0>0$, independent of $h$, such that
		\begin{equation}\label{eq:estimate_boundary_v_1}
			\mathcal B_{\partial\widetilde\Omega}(v_1)\le Ch^{-8},\quad 0<h<h_0.
		\end{equation}
	\end{lemma}
	
	\begin{proof}
		Choose a fixed two-sided neighborhood $U_0\Subset\widehat\Omega_1$ of $\partial\widetilde\Omega$, small enough that $U_0\cap\widetilde\Omega\subset\omega$. Extend $\varphi_2$ smoothly to $U_0$, independently of $h$, keeping its original values on $U_0\cap\widetilde\Omega$, and define $v_1:=e^{-\sigma\varphi_2/h}u_{1,\sigma,h}$ there. Set $\theta:=\sigma(\varphi_1-\varphi_2)$. Since $\varphi_1=\varphi_2$ on $\partial\widetilde\Omega$, we have
		\[
		|\theta(x)|\le C\operatorname{dist}(x,\partial\widetilde\Omega),
		\quad x\in U_0,
		\]
		after reducing $U_0$ if necessary. For a fixed small $c_0>0$, set
		\[
		U_h:=\{x\in U_0:\operatorname{dist}(x,\partial\widetilde\Omega)<c_0h\},
		\]
		then $|\theta|/h$ is uniformly bounded on $U_h$. The CGO representation \eqref{eq:signed_CGO_section4} and the remainder estimates give
		\[
		\|v_1\|_{L^2(U_h)}\le C\bigl(|U_h|^{1/2}+\|r_{\sigma,h}\|_{L^2(\widehat\Omega_1)}
		+\|\widetilde r_{\sigma,h}\|_{L^2(\widehat\Omega_1)}\bigr)\le C.
		\]
		
		In $U_0$, the function $v_1$ satisfies
		\begin{equation*}
			-h^2\Delta v_1-2\sigma h\nabla\varphi_2\cdot\nabla v_1+(h^2q_1-\sigma h\Delta\varphi_2-|\nabla\varphi_2|^2)v_1=0.
		\end{equation*}
		The term $\Delta\varphi_2$ is included because the extension outside $\widetilde\Omega$ need not be harmonic. Choose $0<\gamma<c_0/4$. For $x_0\in\partial\widetilde\Omega$ and small $h$, the ball $B(x_0,2\gamma h)$ is contained in $U_h$. Define $V_{h,x_0}(y):=v_1(x_0+hy)$. In $B(0,2\gamma)$ it solves
		\begin{equation*}
			-\Delta_yV_{h,x_0}-2\sigma\nabla\varphi_2(x_0+hy)\cdot\nabla_yV_{h,x_0}+(h^2q_1-\sigma h\Delta\varphi_2-|\nabla\varphi_2|^2)(x_0+hy)V_{h,x_0}=0.
		\end{equation*}
		The first- and zeroth-order coefficients have uniformly bounded $C^{1,\alpha}$ norms, since $q_1\in C^{1,\alpha}(\overline{\widehat\Omega_1})$ and the extension of $\varphi_2$ is fixed and smooth. The two-dimensional change of variables gives
		\[
		\|V_{h,x_0}\|_{L^2(B(0,2\gamma))}
		=h^{-1}\|v_1\|_{L^2(B(x_0,2\gamma h))}
		\le Ch^{-1}.
		\]
		Choose $p>2/(1-\alpha)$ and put $\beta:=1-2/p>\alpha$. Interior elliptic regularity on nested balls gives an $H^2$ bound by $C\|V_{h,x_0}\|_{L^2(B(0,2\gamma))}$. In dimension two, this also bounds the $W^{1,p}$ norm. The interior $L^p$ estimate then gives a $W^{2,p}$ bound, and Sobolev embedding gives a $C^{1,\beta}$ bound. Writing the equation as $\Delta_yV_{h,x_0}=A_{h,x_0}\cdot\nabla_yV_{h,x_0}+B_{h,x_0}V_{h,x_0}$, the right-hand side has $C^{0,\alpha}$ norm bounded by the same quantity. A Schauder estimate gives a $C^{2,\alpha}$ bound. Since the coefficients have uniform $C^{1,\alpha}$ bounds, the right-hand side now has a $C^{1,\alpha}$ bound, and a second Schauder estimate yields
		\[
		\|V_{h,x_0}\|_{C^{3,\alpha}(\overline{B(0,\gamma)})}
		\le C\|V_{h,x_0}\|_{L^2(B(0,2\gamma))}
		\le Ch^{-1}.
		\]
		All constants are independent of $x_0$ and $h$. Returning to the original variables yields
		\[
		|\partial_x^\mu v_1(x_0)|\le Ch^{-|\mu|-1},
		\quad |\mu|\le3.
		\]
		
		The coefficients of the tangential derivatives and of the normal field on $\partial\widetilde\Omega$ are fixed and smooth. The preceding bounds give $\|v_1\|_{W^{2,\infty}(\partial\widetilde\Omega)}\le Ch^{-3}$ and $\|\partial_\nu v_1\|_{W^{2,\infty}(\partial\widetilde\Omega)}\le Ch^{-4}$. Squaring and using \eqref{eq:def_B_partial_Omega} proves \eqref{eq:estimate_boundary_v_1}.
	\end{proof}
	
	\begin{lemma}\label{lem:polynomial_upper_bound}
		There exist $C>0$ and $h_0>0$, depending on the fixed configurations, phases, and amplitude, such that for every $0<h<h_0$ the following holds. If $\Lambda_{D_1,q_1}f_h=\Lambda_{D_2,q_2}f_h$, the solution from Lemma \ref{lem:collar_gluing} satisfies
		\begin{equation}\label{eq:poly_upper_bound_u2}
			\|e^{-\sigma\varphi_2/h}u_2\|_{L^2(\widetilde\Omega_2)}\le Ch^{-6}.
		\end{equation}
	\end{lemma}
	
	\begin{proof}
		Set $v:=e^{-\sigma\varphi_2/h}u_2$. The glued equation gives
		\[
		\begin{cases}
			e^{-\sigma\varphi_2/h}(-\Delta+q_2)e^{\sigma\varphi_2/h}v=0
			&\text{in }\widetilde\Omega_2,\\
			v=0&\text{on }\partial D_2.
		\end{cases}
		\]
		Lemma \ref{lem:collar_gluing} gives $v\in H^2(\widetilde\Omega_2)$. On an interior collar of $\partial\widetilde\Omega$, it agrees with $v_1$, whose Dirichlet and normal traces belong to $W^{2,\infty}(\partial\widetilde\Omega)$ by Lemma \ref{lem:outer_boundary_polynomial_control}. Hence, Proposition \ref{prop:Carleman_boundary_terms} applies with weight $\sigma\varphi_2$ and $V=q_2$.
		
		The conjugated equation term is zero, and $\mathcal B_{\partial\widetilde\Omega}(v)=\mathcal B_{\partial\widetilde\Omega}(v_1)$. Proposition \ref{prop:Carleman_boundary_terms} with $N=5$ and Lemma \ref{lem:outer_boundary_polynomial_control} give
		\begin{equation*}
			h^{-1}\|v\|_{L^2(\widetilde\Omega_2)}^2\le Ch^{-5}\mathcal B_{\partial\widetilde\Omega}(v)=Ch^{-5}\mathcal B_{\partial\widetilde\Omega}(v_1)\le Ch^{-13}.
		\end{equation*}
		The last inequality gives $\|v\|_{L^2(\widetilde\Omega_2)}^2\le Ch^{-12}$, which proves \eqref{eq:poly_upper_bound_u2}.
	\end{proof}

	\subsection{The oscillatory lower bound and the contradiction}

	\begin{lemma}\label{lem:oscillatory_factor_L2}
		Let $B\Subset\omega$ be a ball such that $d\psi_1\neq0$ on $\overline B$, and let $a$ be holomorphic in a neighborhood of $\overline B$ and not identically zero on $B$. Then, for each $\sigma\in\{1,-1\}$, there exist constants $c_B>0$ and $h_B>0$ such that
		\begin{equation*}
			\|e^{\mathsf{i}\sigma\psi_1/h}a-e^{-\mathsf{i}\sigma\psi_1/h}\overline a\|_{L^2(B)}\ge c_B
		\end{equation*}
		for all $0<h<h_B$.
	\end{lemma}
	
	\begin{proof}
		Put $b_h:=e^{\mathsf i\sigma\psi_1/h}a-e^{-\mathsf i\sigma\psi_1/h}\overline a$. Expanding the squared modulus gives
		\begin{equation*}
			\|b_h\|_{L^2(B)}^2=2\int_B|a|^2\,dx-2\operatorname{Re}\int_Be^{2\mathsf i\sigma\psi_1/h}a^2\,dx.
		\end{equation*}
		Since $|\nabla\psi_1|$ has a positive lower bound on $\overline B$, the coarea formula gives
		\begin{equation*}
			\int_Be^{2\mathsf i\sigma\psi_1/h}a^2\,dx=\int_{\mathbb R}e^{2\mathsf i\sigma t/h}A(t)\,dt,\quad A(t):=\int_{B\cap\{\psi_1=t\}}\frac{a^2}{|\nabla\psi_1|}\,d\mathcal H^1.
		\end{equation*}
		The function $A$ is defined for almost every $t$ and is set to zero outside $\psi_1(B)$. A second application of coarea gives $\|A\|_{L^1(\mathbb R)}\le\int_B|a|^2\,dx$. The Riemann--Lebesgue lemma makes the oscillatory integral tend to zero as $h\downarrow0$, for either sign $\sigma$. In particular, $\|b_h\|_{L^2(B)}^2\to2\int_B|a|^2\,dx>0$. One may take $c_B=(\int_B|a|^2\,dx)^{1/2}$ after decreasing $h_B$.
	\end{proof}
	
	\begin{lemma}\label{lem:phase_gap_exclusion}
		There exists $h_*>0$ such that, for every $0<h<h_*$, one cannot have
		\begin{equation*}
			\Lambda_{D_1,q_1}f_h=\Lambda_{D_2,q_2}f_h.
		\end{equation*}
	\end{lemma}
	
	\begin{proof}
		Take $h>0$ small enough for the preceding constructions and estimates, and suppose that $\Lambda_{D_1,q_1}f_h=\Lambda_{D_2,q_2}f_h$. Let $u_2$ be the solution extended by Lemma \ref{lem:collar_gluing}. Lemma \ref{lem:polynomial_upper_bound} gives
		\begin{equation}\label{eq:upper_poly_section4}
			\|e^{-\sigma\varphi_2/h}u_2\|_{L^2(\widetilde\Omega_2)}\le Ch^{-6},
		\end{equation}
		where $C$ is independent of this choice of $h$. In $B\Subset\omega$, the glued solution equals $u_{1,\sigma,h}$. By \eqref{eq:signed_CGO_section4},
		\begin{equation*}
			e^{-\sigma\varphi_2/h}u_2=e^{\sigma(\varphi_1-\varphi_2)/h}(b_h+R_h),\quad b_h:=e^{\mathsf i\sigma\psi_1/h}a_\sigma-e^{-\mathsf i\sigma\psi_1/h}\overline{a_\sigma},
		\end{equation*}
		where $R_h:=e^{\mathsf i\sigma\psi_1/h}r_{\sigma,h}-e^{-\mathsf i\sigma\psi_1/h}\overline{r_{\sigma,h}}+\widetilde r_{\sigma,h}$. The remainder bounds give $\|R_h\|_{L^2(B)}\le2\|r_{\sigma,h}\|_{L^2(\widehat\Omega_1)}+\|\widetilde r_{\sigma,h}\|_{L^2(\widehat\Omega_1)}=o(1)$.
		
		Lemma \ref{lem:oscillatory_factor_L2} gives $\|b_h\|_{L^2(B)}\ge c_B$ for every sufficiently small $h$. Decrease the upper bound on $h$ so that also $\|R_h\|_{L^2(B)}\le c_B/2$. The triangle inequality gives $\|b_h+R_h\|_{L^2(B)}\ge c_B/2$. Since $\sigma(\varphi_1-\varphi_2)\ge2c$ throughout $B$, multiplication by this real exponential gives
		\begin{equation}\label{eq:lower_exp_section4}
			\|e^{-\sigma\varphi_2/h}u_2\|_{L^2(B)}\ge e^{2c/h}\|b_h+R_h\|_{L^2(B)}\ge\frac{c_B}{2}e^{2c/h}.
		\end{equation}
		The inclusion $B\subset\widetilde\Omega_2$ and \eqref{eq:upper_poly_section4} imply $(c_B/2)e^{2c/h}\le Ch^{-6}$. Since $h^6e^{2c/h}\to\infty$, this is impossible for all sufficiently small $h$. Choose $h_*>0$ below this threshold and those of the preceding constructions. Equality of the Neumann data is excluded for every $0<h<h_*$.
	\end{proof}

	\begin{proof}[Proof of Theorem \ref{thm:obstacle_uniqueness}]
		Assume $D_1\ne D_2$. Choose $\sigma$, $B$, and $c$ by Lemma \ref{lem:phase_gap}, and form the CGO family above with $f_h:=u_{1,\sigma,h}|_{\partial\Omega}$. Lemma \ref{lem:phase_gap_exclusion} gives $\Lambda_{D_1,q_1}f_h\ne\Lambda_{D_2,q_2}f_h$, for every sufficiently small $h$. Both DN maps send zero to zero, so every such $f_h$ is nonzero. This proves the assertion.
	\end{proof}
	
	\begin{proof}[Proof of Theorem \ref{thm:Schiffer}]
		If $D_1=D_2$, the assertion holds for any nonzero boundary value. If $D_1\ne D_2$, Theorem \ref{thm:obstacle_uniqueness} supplies a nonzero CGO trace $f$ for which $\Lambda_{D_1,q_1}f\ne\Lambda_{D_2,q_2}f$. For this choice of $f$, equality of the Neumann data can therefore hold only when $D_1=D_2$.
	\end{proof}

	\section{Unique determination of the potential}\label{sec:potential}

	\begin{proof}[Proof of Corollary \ref{cor:simultaneous_recovery}]
		By Theorem \ref{thm:Schiffer}, \eqref{eq:same_DN} implies $D_1=D_2=:D$. Set $\mathcal O:=\Omega\setminus\overline D$. Equality of the outer DN maps gives equality of the partial Cauchy data on $\mathcal O$, with measurements on $\partial\Omega$ and zero Dirichlet values on $\partial D$. Since $q_1,q_2\in C^{1,\alpha}(\overline{\mathcal O})$ are real-valued, \cite[Theorem 1.1]{guillarmou2011calderon} gives $q_1=q_2$ in $\mathcal O$.
	\end{proof}

	\section*{Statements and Declarations}
	
	\para{Data availability statement}
	No datasets were generated or analyzed during the current study.
	
	\para{Conflict of interest}
	The authors declare that they have no conflict of interest.

	\para{Acknowledgments} Y.-H. L. is partially supported by the National Science and Technology Council (NSTC) of Taiwan, under the project 113-2628-M-A49-003. Y.-H. L. acknowledges financial support from the Alexander von Humboldt Foundation through the Henriette Herz Scouting Programme and is hosted by Universität Duisburg-Essen. T.L. is partly supported by the Research Council of Finland (Centre of Excellence in Inverse Modelling and Imaging and FAME Flagship, grants 353091 and 359208).

	\bibliography{ref} 
	
	\bibliographystyle{alpha}
	
\end{document}